\documentclass[11pt, a4paper]{amsart}
\usepackage[latin1]{inputenc}
\usepackage[english]{babel}
\usepackage{amsmath, amsfonts, amssymb, amsthm}
\usepackage{enumerate}
\usepackage[numbers]{natbib}
\usepackage{setspace}

\usepackage[arrow,matrix,curve]{xy}
\usepackage[pagebackref,colorlinks,linkcolor=black,citecolor=blue,urlcolor=blue,hypertexnames=true]{hyperref}

\usepackage{tikz}
\usetikzlibrary{matrix}

\newtheorem{theorem}{Theorem}[section]
\newtheorem{lemma}[theorem]{Lemma}
\newtheorem{corollary}[theorem]{Corollary}

\theoremstyle{definition}
\newtheorem{definition}[theorem]{Definition}

\newtheorem{remark}[theorem]{Remark}

\newcommand{\N}{\mathbb{N}}
\newcommand{\Z}{\mathbb{Z}}

\newcommand{\R}{\mathbb{R}}
\newcommand{\C}{\mathbb{C}}
\newcommand{\K}{\mathbb{K}}		
\newcommand{\I}{\mathcal{I}}    

\newcommand{\Cuntz}[1]{\mathcal{O}_{#1}}
\newcommand{\Aut}[1]{{\rm Aut}(#1)}

\newcommand{\uAut}[1]{{\rm Aut}_0(#1)}
\newcommand{\id}[1]{{\rm id}_{#1}}

\newcommand{\obj}{{\rm obj}}
\newcommand{\mor}{{\rm mor}}
\newcommand{\Proj}[1]{\mathcal{P}r(#1)}

\newcommand{\units}[1]{\Omega^{\infty}(#1)^*}

\DeclareMathOperator{\hocolim}{\rm hocolim}
\DeclareMathOperator{\colim}{\rm colim}
\DeclareMathOperator{\Tel}{\rm Tel}

\newcommand{\Bun}[1]{\mathcal{B}un_X(#1)}

\begin{document}
	
\title{Unit spectra of $K$-theory \\ from strongly self-absorbing $C^*$-algebras}
\author{Marius Dadarlat \and Ulrich Pennig}	
\thanks{M.D. was partially supported by NSF grant \#DMS--1101305}
\begin{abstract}
We give an operator algebraic model for the first group of the unit spectrum $gl_1(KU)$ of complex topological $K$-theory, i.e.\ $[X, BGL_1(KU)]$, by bundles of stabilized infinite Cuntz $C^{\ast}$-algebras $\Cuntz{\infty} \otimes \K$. We develop similar models for the localizations of $KU$ at a prime $p$ and away from $p$. Our work is based on the $\I$-monoid model for the units of $K$-theory by Sagave and Schlichtkrull and it was motivated by 
the goal of finding connections between the infinite loop space structure of the classifying space of the automorphism group of stabilized strongly self-absorbing $C^{\ast}$-algebras  that arose in our generalization  of the  Dixmier-Douady theory  and classical spectra from algebraic topology.

\end{abstract}

\maketitle
\tableofcontents

\section{Introduction}

Suppose $E^{\bullet}$ is a multiplicative generalized cohomology theory represented by a commutative ring spectrum $R$. The units $GL_1(E^0(X))$ of $E^0(X)$ provide an abelian group functorially associated to the space $X$. From the point of view of algebraic topology it is therefore a natural question, whether we can lift $GL_1$ to spectra, i.e.\ whether there is a spectrum of units $gl_1(R)$ such that $gl_1(R)^0(X) = GL_1(E^0(X))$.

It was realized by Sullivan in \cite{book:GeometricTopology} that $gl_1(R)$ is closely connected to questions of orientability in algebraic topology. In particular, the units of $K$-theory act on the $K$-orientations of PL-bundles. Segal \cite{paper:SegalCatAndCoh} proved that the classifying space $\{1\} \times BU \subset \Z \times BU$ for virtual vector bundles of virtual dimension $1$ equipped with the $H$-space structure from the tensor product is in fact a $\Gamma$-space, which in turn yields a spectrum of a connective generalized cohomology  theory $bu^*_\otimes(X)$. His method is easily extended to include the virtual vector bundles of virtual dimension $-1$ to obtain a generalized cohomology theory $gl_1(KU)^*(X) \supset bu^*_{\otimes}(X)$ answering the above question affirmatively: $ GL_1(K^0(X)) \cong gl_1(KU)^0(X)$. Later May, Quinn, Ray and Tornehave \cite{book:MayEInfty} came up with the notion of $E_{\infty}$-ring spectra, which always have associated unit spectra.

Since $gl_1(R)$ is defined via stable homotopy theory, there is in general no nice geometric interpretation of the higher groups even though $R$ may have one. In particular, no geometric interpretation was known for $gl_1(KU)^k(X)$. In this article we give an operator algebra interpretation of $gl_1(KU)^1(X)$ as the group of isomorphism classes of locally trivial bundles of C*-algebras with fiber isomorphic to the stable Cuntz algebra $\mathcal{O}_\infty \otimes \K$ with the group operation induced by the tensor product. 
In fact one can also recover Segal's original infinite loop space $BBU_\otimes$ as $B\Aut{\mathcal{Z} \otimes \K}$,
where $\mathcal{Z}$ is the ubiquitous Jiang-Su algebra \cite{paper:TomsWinter}.
For localizations of $KU$ we obtain that $gl_1(KU_{(p)})^1(X)$ is the group of isomorphism classes of locally trivial bundles with fiber isomorphic to the C*-algebra $M_{(p)}\otimes \mathcal{O}_\infty \otimes \K$ with the group operation induced by the tensor product. 
Here $M_{(p)}$ is a $C^*$-algebra with $K_0(M_{(p)}) \cong \Z_{(p)}$, $K_1(M_{(p)}) = 0$ that can be obtained as an infinite tensor product of matrix algebras.

Our approach is based on the work of Sagave and Schlichtkrull \cite{paper:SagaveSchlichtkrull, paper:Schlichtkrull}, who developed a representation of $gl_1(R)$ for a commutative symmetric ring spectrum $R$ as a commutative $\I$-monoid. Motivated by the definition of twisted cohomology theories, we study the following situation, which appears to be a natural setup beyond   the case where $R$ is $K$-theory: Suppose $G$ is an $\I$-space, such that each $G(\mathbf{n})$ is a topological group acting on $R_n$. To formulate a sensible compatibility condition between the group action $\kappa$ and the multiplication $\mu^R$ on $R$, we need to demand that $G$ itself carries an additional $\I$-monoid structure $\mu^G$ and the following diagram commutes:
\[
	\xymatrix{
	G(\mathbf{m}) \times R_m \times G(\mathbf{n}) \times R_n \ar[rr]^-{\kappa_m \times \kappa_n} \ar[d]_-{(\mu^G_{m,n} \times \mu^R_{m,n}) \circ \tau} & &  R_m \times R_n \ar[d]^-{\mu^R_{m,n}}\\
	G(\mathbf{m \sqcup n}) \times R_{m+n} \ar[rr]_-{\kappa_{m+n}}& & R_{m+n}
	}
\] 
where $\tau$ switches the two middle factors. Associativity of the group action suggests that the analogous diagram, which has $G(\mathbf{n})$ in place of $R_n$ and $\mu^G$ instead of $\mu^R$, should also commute. This condition can be seen as a homotopy theoretic version of the property needed for the Eckmann-Hilton trick, which is why we will call such a $G$ an Eckmann-Hilton $\I$-group (EH-$\I$-group for short). Commutativity of the above diagram has the following important implications:
\pagebreak
\begin{itemize}
	\item the $\I$-monoid structure of $G$ is commutative (Lemma \ref{lem:commutative}),
	\item the classifying spaces $B_{\nu}G(\mathbf{n})$ with respect to the group multiplication $\nu$ of $G$ form a commutative $\I$-monoid $\mathbf{n} \mapsto B_{\nu}G(\mathbf{n})$,
	\item if $G$ is convergent and $G(\mathbf{n})$ has the homotopy type of a CW-complex, then the $\Gamma$-spaces associated to $G$ and $B_{\nu}G$ satisfy $B_{\mu}\Gamma(G) \simeq \Gamma(B_{\nu}G)$, where $B_{\mu}\Gamma(X)$ for a commutative $\I$-monoid $X$ denotes the $\Gamma$-space delooping of $\Gamma(X)$ (Theorem \ref{thm:EHI_group}).
\end{itemize}
Let $\units{R}(\mathbf{n})$ be the commutative $\I$-monoid with associated spectrum $gl_1(R)$. If $G$ acts on $R$ and the inverses of $G$ with respect to both multiplicative structures $\mu^G$ and $\nu$ are compatible in the sense of Definition \ref{def:EHI_group}, then the action induces a map of $\Gamma$-spaces $\Gamma(G) \to \Gamma(\units{R})$. This deloops to a map $B_{\mu}\Gamma(G) \to B_{\mu}\Gamma(\units{R})$ and we give sufficient conditions for this to be a strict equivalence of (very special) $\Gamma$-spaces.

In the second part of the paper, we consider the EH-$\I$-group $G_A(\mathbf{n}) = \Aut{(A \otimes \K)^{\otimes n}}$ associated to the automorphisms of a (stabilized) strongly self-absorbing $C^*$-algebra $A$. This class of $C^*$-algebras was introduced by Toms and Winter in \cite{paper:TomsWinter}. It contains the algebras $\Cuntz{\infty}$ and $M_{(p)}$ alluded to above as well as the Jiang-Su algebra $\mathcal{Z}$ and the Cuntz algebra $\Cuntz{2}$. It is closed with respect to tensor product and plays a fundamental role in the classification theory of nuclear $C^*$-algebras. 

For a strongly self-absorbing $C^*$-algebra $A$, $X \mapsto K_0(C(X) \otimes A)$ turns out to be a multiplicative cohomology theory. In fact, this structure can be lifted to a commutative symmetric ring spectrum $KU^A$ along the lines of \cite{paper:HigsonGuentner, paper:Joachim, paper:FunctorialKSpectrum}. The authors showed in \cite{paper:DadarlatP1} that $B\Aut{A \otimes \K}$ is an infinite loop space and the first space in the spectrum of a generalized cohomology theory $E_A^*(X)$ such that $E_A^0(X) \cong K_0(C(X) \otimes A)^{\times}_+$, in particular $E_{\Cuntz{\infty}}^0(X) \cong GL_1(K^0(X))$, which suggest that $E_{\Cuntz{\infty}}^*(X) \cong gl_1(KU)^*(X)$. In fact, we can prove:
\begin{theorem}
Let $A \neq \C$ be a separable strongly self-absorbing $C^*$-algebra. 
\begin{enumerate}[{\rm (a)}]
	\item The EH-$\I$-group $G_A$ associated to $A$ acts on the commutative symmetric ring spectrum $KU^A_{\bullet}$ inducing a map $\Gamma(G_A) \to \Gamma(\units{KU^A})$.
	\item The induced map on spectra is an isomorphism on all homotopy groups $\pi_n$ with $n > 0$ and the inclusion $K_0(A)^{\times}_+ \to K_0(A)^{\times}$ on $\pi_0$. 
	\item In particular $B\Aut{A \otimes \Cuntz{\infty} \otimes \K} \simeq BGL_1(KU^A)$ and 
	\(
		gl_1(KU)^1(X) \cong \Bun{A \otimes \Cuntz{\infty} \otimes \K}
	\),
	where the right hand side denotes the group of isomorphism classes of $C^*$-algebra bundles with fiber $A \otimes \Cuntz{\infty} \otimes \K$ with respect to the tensor product $\otimes$.
\end{enumerate}

\end{theorem}

We also compare the spectrum defined by the $\Gamma$-space $\Gamma(B_{\nu}G_A)$ with the one obtained from the infinite loop space construction used in \cite{paper:DadarlatP1} and show that they are equivalent. The group $gl_1(KU^A)^1(X)$ alias $E_A^1(X)$ is a natural receptacle for invariants of not necessarily locally trivial continuous fields of $C^*$-algebras with stable strongly self-absorbing fibers that satisfy a Fell condition. This provides a substantial extension of results by Dixmier and Douady with $gl_1(KU^A)$ replacing ordinary cohomology, \cite{paper:DadarlatP1}.
The above theorem lays the ground for an operator algebraic interpretation of the ``higher'' twists of $K$-theory. Twisted $K$-theory as defined first by Donovan and Karoubi \cite{paper:DonovanKaroubi} and later in increased generality by Rosenberg \cite{paper:Rosenberg} and Atiyah and Segal \cite{paper:AtiyahSegal} has a nice interpretation in terms of bundles of compact operators \cite{paper:Rosenberg}. From the point of view of homotopy theory, it is possible to define twisted $K$-theory with more than just the $K(\Z,3)$-twists \cite{paper:Ando, book:MaySigurdsson} and the present paper suggests an interpretation of these more general invariants in terms of bundles with fiber $\Cuntz{\infty}\otimes \K$. We will pursue this idea in upcoming work.\\ 

\paragraph{\it Acknowledgements}
The second named author would like to thank Johannes Ebert for many useful discussions and Tyler Lawson for an answer to his question on mathoverflow.net. The first named author would like to thank Jim McClure for a helpful
discussion on localization of spectra. 
 
\section{Preliminaries}

\subsection{Symmetric ring spectra, units and $\I$-spaces}
Since our exposition below is based on symmetric ring spectra and their units, we will recall their definition in this section. The standard references for this material are \cite{paper:HoveyShipleySmith} and \cite{paper:MandellMaySchwedeShipley}. Let $\Sigma_n$ be the symmetric group on $n$ letters and let $S^n = S^1 \wedge \dots \wedge S^1$ be the smash product of $n$ circles. This space carries a canonical $\Sigma_n$-action. Define $S^0$ to be the two-point space. Let $\mathcal{T}op$ be the category of compactly generated Hausdorff spaces and denote by $\mathcal{T}op_*$ its pointed counterpart. 

\begin{definition}\label{def:symmetric_spectrum}
A \emph{commutative symmetric ring spectrum} $R_{\bullet}$ consists of a sequence of pointed topological spaces $R_n$ for $n \in \N_0 = \{0,1,2, \dots\}$ with a basepoint preserving action by $\Sigma_n$ together with a sequence of pointed equivariant maps $\eta_n \colon S^n \to R_n$ and a collection $\mu_{m,n}$ of pointed $\Sigma_m \times \Sigma_n$-equivariant maps 
\[
	\mu_{m,n} \colon R_m \wedge R_n \to R_{m+n}
\]
such that the following conditions hold:
\begin{enumerate}[(a)]
	\item associativity: $\mu_{p+q,r} \circ (\mu_{p,q} \wedge \id{R_r}) = \mu_{p,q+r} \circ (\id{R_p} \wedge \mu_{q,r})$,
	\item compatibility: $\mu_{p,q} \circ (\eta_p \wedge \eta_q) = \eta_{p+q}$
	\item commutativity: The following diagram commutes
	\[
		\xymatrix{
			R_m \wedge R_n \ar[r]^-{\mu_{m,n}} \ar[d]_-{\rm tw} & R_{m+n} \ar[d]^-{\tau_{m,n}} \\
			R_n \wedge R_m \ar[r]_-{\mu_{n,m}} & R_{n+m}
		}
	\]
	where ${\rm tw}$ is the flip map and $\tau_{m,n}$ is the block permutation exchanging the first $m$ letters with the last $n$ letters preserving their order. 
\end{enumerate}
\end{definition}

In order to talk about units in a symmetric ring spectrum with respect to its graded multiplication $\mu_{m,n} \colon R_m \wedge R_n \to R_{m+n}$ we need to deal with homotopy colimits. Given a small (discrete) indexing category $\mathcal{J}$ and a functor $F \colon \mathcal{J} \to \mathcal{T}op$, i.e.\ a diagram in spaces, we define its \emph{homotopy colimit} $\hocolim_{\mathcal{J}} F$ to be the geometric realization of its (topological) \emph{transport category} $\mathcal{T}_F$ with object space $\coprod_{j \in \obj(\mathcal{J})} F(j)$ and morphism space $\coprod_{j,j' \in \obj(\mathcal{J})} F(j) \times \hom_{\mathcal{J}}(j,j')$, where the source map is given by the projection to the first factor and the value of the target map on a morphism $(x,f) \in F(j) \times \hom_{\mathcal{J}}(j,j')$ is $F(f)(x)$ \cite[Proposition 5.7]{paper:Weiss}. 


To define inverses for the graded multiplication $\mu$ of $R$, we need a bookkeeping device that keeps track of the degree, i.e.\ the suspension coordinate. We follow the work of Sagave and Schlichtkrull \cite{paper:SagaveSchlichtkrull, paper:Schlichtkrull}, in particular \cite[section 2.2]{paper:Schlichtkrull}, which tackles this issue using $\I$-spaces and $\I$-monoids.

\begin{definition} \label{def:ISpace_IMonoid} 
Let $\I$ be the category, whose objects are the finite sets $\mathbf{n} = \{1, \dots, n\}$ and whose morphisms are injective maps. The empty set $\mathbf{0}$ is an initial object in this category. Concatenation $\mathbf{m} \sqcup \mathbf{n}$ and the block permutations $\tau_{m,n} \colon \mathbf{m} \sqcup \mathbf{n} \to \mathbf{n} \sqcup \mathbf{m}$ turn $\I$ into a symmetric monoidal category. An \emph{$\I$-space} is a functor $X \colon \I \to \mathcal{T}op_{\ast}$.

Moreover, an $\I$-space $X$ is called an \emph{$\I$-monoid} if it comes equipped with a natural transformation $\mu \colon X \times X \to X \circ \sqcup$ of functors $\I^2 \to \mathcal{T}op_{\ast}$, i.e.\ a family of continuous maps 
\[
	\mu_{m,n} \colon X(\mathbf{m}) \times X(\mathbf{n}) \to X(\mathbf{m} \sqcup \mathbf{n})\ ,
\]
which is associative in the sense that $\mu_{l,m+n} \circ (\id{X(\mathbf{l})} \times \mu_{m,n}) = \mu_{l+m, n} \circ (\mu_{l,m} \times \id{X(\mathbf{n})})$ for all $\mathbf{l}, \mathbf{m}, \mathbf{n} \in {\rm obj}(\I)$ and unital in the sense that the diagrams
\[
	\xymatrix{
		X(\mathbf{0}) \times X(\mathbf{n}) \ar[r]^-{\mu_{0,n}} & X(\mathbf{n}) \\
		X(\mathbf{n}) \ar[u] \ar[ur]_{\id{X(\mathbf{n})}}
	}
	\qquad
	\xymatrix{
		X(\mathbf{n}) \times X(\mathbf{0}) \ar[r]^-{\mu_{n,0}} & X(\mathbf{n}) \\
		X(\mathbf{n}) \ar[u] \ar[ur]_{\id{X(\mathbf{n})}}
	}
\]
commute for every $\mathbf{n} \in {\rm obj}(\I)$, where the two upwards arrows are the inclusion with respect to the basepoint in $X(\mathbf{0})$. Likewise we call an $\I$-monoid $X$ \emph{commutative}, if 
\[
	\xymatrix{
		X(\mathbf{m}) \times X(\mathbf{n}) \ar[d]_-{\text{sw}} \ar[r]^-{\mu_{m,n}} & X(\mathbf{m} \sqcup \mathbf{n}) \ar[d]^{\tau_{m,n*}} \\
		X(\mathbf{n}) \times X(\mathbf{m}) \ar[r]^-{\mu_{n,m}} & X(\mathbf{n} \sqcup \mathbf{m})
	}
\]
commutes. We denote the homotopy colimit of $X$ over $\I$ by $X_{h\I} = \hocolim_{\I}(X)$. If $X$ is an $\I$-monoid, then $(X_{h\I},\mu)$ is a topological monoid as explained in \cite[p.\ 652]{paper:Schlichtkrull}. We call $X$ \emph{grouplike}, if $\pi_0(X_{h\I})$ is a group with respect to the multiplication induced by the monoid structure. 
\end{definition}

Let $\mathcal{N}$ be the category associated to the directed poset $\N_0=\{0,1,2,...\}$. Note that there is an inclusion functor $\mathcal{N} \to \I$, which sends a map $n \to m$ to the standard inclusion $\mathbf{n} \to \mathbf{m}$. This way, we can associate a space called the \emph{telescope} to an $\I$-space $X$: $\Tel(X) := \hocolim_{\mathcal{N}} X$. If $X$ is convergent, then $\Tel(X) \to X_{h\I}$ is a weak homotopy equivalence. Any space $Y$ together with a continuous self-map $f \colon Y \to Y$ yields a functor $F \colon \mathcal{N} \to \mathcal{T}op$ with $F(n) = Y$, $F(n \to m) = f^{(m-n)}$, where $f^{(m-n)}$ denotes the composition of $m-n$ copies of $f$. We will denote the associated telescope by $\Tel(Y;f)$ or $\Tel(Y)$ if the map is clear.

As described in \cite[section 2.3]{paper:Schlichtkrull}, a symmetric ring spectrum $R_{\bullet}$ yields an $\I$-monoid $\Omega^{\infty}(R)$ as follows: Let $\Omega^{\infty}(R)(\mathbf{n}) = \Omega^n(R_n)$ with basepoint $\eta_n$. A morphism $\alpha \colon \mathbf{m} \to \mathbf{n}$ uniquely defines a permutation $\bar{\alpha} \colon \mathbf{n} = \mathbf{l} \sqcup \mathbf{m} \to \mathbf{n}$, which is order preserving on the first $l$ elements and given by $\alpha$ on the last $m$ entries. Mapping $f \in \Omega^m(R_m)$ to 
\[
	\xymatrix{
	S^n  \ar[r]^-{\bar{\alpha}^{-1}} & S^n  \ar[r]^-{\eta_l \wedge f} & R_l \wedge R_m \ar[r]^-{\mu_{l,m}} & R_{n} \ar[r]^-{\bar{\alpha}} & R_n\ .
	}
\]
yields the functoriality with respect to injective maps. The monoid structure is induced by the multiplication of $R_{\bullet}$ as follows
\[
	\mu_{m,n}(f,g) \colon 
	\xymatrix{
	S^m \wedge S^n \ar[r]^-{f \wedge g} & R_m \wedge R_n \ar[r]^-{\mu_{m,n}} & R_{m+n}
	}
\]
for $f \in \Omega^m(R_m)$ and $g \in \Omega^n(R_n)$. If $R$ is commutative, then $\Omega^{\infty}(R)$ is a commutative $\I$-monoid. 

Let $\units{R}$ be the $\I$-monoid of units of $R$ given as follows $\units{R}(\mathbf{n})$ is the union of those components of $\Omega^{\infty}(R)(\mathbf{n})$ that have stable inverses in the sense that for each $f \in \units{R}(\mathbf{n})$ there exists $g \in \Omega^{\infty}(R)(\mathbf{m})$, such that $\mu_{n,m}(f,g)$ and $\mu_{m,n}(g,f)$ are homotopic to the basepoint of $\Omega^{\infty}(R)(\mathbf{n} \sqcup \mathbf{m})$ and  $\Omega^{\infty}(R)(\mathbf{m} \sqcup \mathbf{n})$ respectively. Define the space of units by
\(
	GL_1(R) = \hocolim_{\I}(\units{R})
\).

If $R$ is commutative, the spectrum of units associated to the $\Gamma$-space $\Gamma(\units{R})$ will be denoted by $gl_1(R)$. If $R$ is convergent, then $\pi_0((\units{R})_{h\I}) = GL_1(\pi_0(R))$, $\Gamma(\units{R})$ is very special and $gl_1(R)$ is an $\Omega$-spectrum.

\section{Eckmann-Hilton $\I$-groups}
As motivated in the introduction, we study the following particularly nice class of $\I$-monoids.

\begin{definition} \label{def:EHI_group}
Let $\mathcal{G}rp_{\ast}$ be the category of topological groups (which we assume to be \emph{well-pointed} by the identity element) and continuous homomorphisms. A functor $G \colon \I \to \mathcal{G}rp_{\ast}$ is called an \emph{$\I$-group}. An $\I$-group $G$ is called an \emph{Eckmann-Hilton $\I$-group} (or EH-$\I$-group for short) if it is an $\I$-monoid in $\mathcal{G}rp_*$ with multiplication $\mu_{m,n}$, such that the following diagram of natural transformations between functors $\I^2 \to \mathcal{G}rp_{\ast}$ commutes, 
\[
	\xymatrix{
		G(\mathbf{m}) \times G(\mathbf{m}) \times G(\mathbf{n}) \times G(\mathbf{n}) \ar[rr]^-{(\mu_{m,n} \times \mu_{m,n}) \circ \tau} \ar[d]_-{\nu_{m} \times \nu_n} && G(\mathbf{m} \sqcup \mathbf{n}) \times G(\mathbf{m} \sqcup \mathbf{n}) \ar[d]^-{\nu_{m+n}} \\
		G(\mathbf{m}) \times G(\mathbf{n}) \ar[rr]_-{\mu_{m,n}} && G(\mathbf{m} \sqcup \mathbf{n})
	}
\]
(where $\nu_n \colon G(\mathbf{n}) \times G(\mathbf{n}) \to G(\mathbf{n})$ denotes the group multiplication and $\tau$ is the map that switches the two innermost factors). We call $G$ \emph{convergent}, if it is convergent as an $\I$-space in the sense of \cite[section 2.2]{paper:Schlichtkrull}. If all morphisms in $\I$ except for the maps $\mathbf{0} \to \mathbf{n}$ are mapped to homotopy equivalences, the EH-$\I$-group $G$ is called \emph{stable} (this implies convergence).

Let $\iota_m \colon \mathbf{0} \to \mathbf{m}$ be the unique morphism. We say that an EH-$\I$-group has \emph{compatible inverses}, if there is a path from $(\iota_m \sqcup \id{\mathbf{m}})_*(g) \in G(\mathbf{m} \sqcup \mathbf{m})$ to $(\id{\mathbf{m}} \sqcup \iota_m)_*(g)$ for all $\mathbf{m} \in \obj(\I)$ and $g \in G(\mathbf{m})$.
\end{definition}

The above diagram is easily recognized as a graded version of the Eckmann-Hilton compatibility condition, where the group multiplication and the monoid structure provide the two operations. Thus, the following comes as no surprise.

\begin{lemma}\label{lem:commutative}
Let $G$ be an EH-$\I$-group. Then the $\I$-monoid structure of $G$ is commutative. 
\end{lemma}

\begin{proof}\label{pf:commutative}
Let $1_m \in G(\mathbf{m})$ be the identity element. If $\iota_{m} \colon \mathbf{0} \to \mathbf{m}$ is the unique morphism, then $\iota_{m*}(1_0) = 1_m$. For $g \in G(\mathbf{n})$
\[
	\mu_{m,n}(1_m,g) = \mu_{m,n}(\iota_{m*}(1_0),g) = (\iota_{m} \sqcup \id{\mathbf{n}})_* \mu_{0,n}(1_0,g) = (\iota_{m} \sqcup \id{\mathbf{n}})_*(g)
\]
by naturality. Let $g \in G(\mathbf{n})$, $h \in G(\mathbf{m})$, then 
\begin{align*}
	\mu_{m,n}(g,h) & = \mu_{m,n}(\nu_m(1_m,g), \nu_n(h,1_n)) = \nu_{m+n}(\mu_{m,n}(1_m, g), \mu_{m,n}(h,1_n)) \\
	& = \nu_{m+n}((\iota_{m} \sqcup \id{\mathbf{n}})_*(g), (\id{\mathbf{m}} \sqcup \iota_n)_*(h)) \\
	& = \nu_{m+n}(\tau_{m,n*}(\id{\mathbf{n}} \sqcup \iota_{m})_*(g), \tau_{m,n*}(\iota_n \sqcup \id{\mathbf{m}})_*(h)) \\
	& = \tau_{m,n*} \nu_{n+m}((\id{\mathbf{n}} \sqcup \iota_{m})_*(g), (\iota_n \sqcup \id{\mathbf{m}})_*(h)) = \tau_{m,n*} \mu_{n,m}(h,g)
\end{align*}
In the last step we have used the fact that $\tau_{m,n*}$ is a group homomorphism.
\end{proof}

\begin{lemma} \label{lem:h_inverse}
Let $G$ be an EH-$\I$-group with compatible inverses, let $g \in G(\mathbf{m})$, then there is a path connecting $\mu_{m,m}(g,g^{-1}) \in G(\mathbf{m} \sqcup \mathbf{m})$ and $1_{m \sqcup m} \in G(\mathbf{m} \sqcup \mathbf{m})$.
\end{lemma}

\begin{proof}\label{pf:h_inverse}
Just as in the proof of Lemma \ref{lem:commutative} we see that
\[
\mu_{m,m}(g,g^{-1}) = \nu_{m+m}((\iota_{m} \sqcup \id{\mathbf{m}})_*(g), (\id{\mathbf{m}} \sqcup \iota_m)_*(g^{-1}))\ .
\]
But by assumption $(\id{\mathbf{m}} \sqcup \iota_m)_*(g^{-1})$ is homotopic to $(\iota_m \sqcup \id{\mathbf{m}})_*(g^{-1})$. We get
\[
	\nu_{m+m}((\iota_{m} \sqcup \id{\mathbf{m}})_*(g), (\iota_m \sqcup \id{\mathbf{m}})_*(g^{-1})) = (\iota_{m} \sqcup \id{\mathbf{m}})_* (\nu_{m}(g,g^{-1})) = (\iota_{m} \sqcup \id{\mathbf{m}})_*(1_m)
\]
proving the claim.
\end{proof}

Let $G$ be an EH-$\I$-group and let $\Delta_r \colon \I \to \I^r$ be the multidiagonal functor that maps $\mathbf{n}$ to $(\mathbf{n}, \dots, \mathbf{n}) \in \obj(\I^r)$. Then we obtain an $\I$-monoid $G^{(r)}$ from this via
\[
	G^{(r)}(\mathbf{n}) = \underbrace{(G \times \dots \times G)}_{r \text{ times}} \circ \Delta_r(\mathbf{n}) = G(\mathbf{n})^r\ ,
\] 
where we define $G^{(0)}(\mathbf{n})$ to be the trivial group and $G^{(1)}(\mathbf{n}) = G(\mathbf{n})$. The multiplication is given by 
\[
	\mu^{(r)}_{m,n} \colon G(\mathbf{m})^r \times G(\mathbf{n})^r \to (G(\mathbf{m}) \times G(\mathbf{n}))^r \to G(\mathbf{m} \sqcup \mathbf{n})^r
\]
where the first map is reshuffling the factors in an order-preserving way and the second map is $\mu_{m,n} = \mu^{(1)}_{m,n}$. We can rephrase the fact that $\varphi_* \colon G(\mathbf{m}) \to G(\mathbf{n})$ is a group homomorphism for every $\varphi \in \mor(\I)$ by saying that $\nu_{\bullet} \colon (G \times G) \circ \Delta_2 \to G$ is a natural transformation. Thus, we obtain face maps of the form
\[
	d_{i,n} \colon G^{(r)}(\mathbf{n}) \to G^{(r-1)}(\mathbf{n}) \quad ; \quad (g_1, \dots, g_r) \mapsto \begin{cases}
	(g_2, \dots, g_r) & \text{if } i = 0 \\
	(g_1, \dots, \nu_n(g_i, g_{i+1}), \dots, g_r) & \text{if } 0 < i < r \\
	(g_1, \dots, g_{r-1}) & \text{if } i = r 
\end{cases}
\]
and corresponding degeneracy maps $s_{i,n} \colon G^{(r)} \to G^{(r+1)}$, which insert the identity of $G(\mathbf{n})$ after the $i$th element. Altogether, we see that $G^{(r)}$ is a simplicial $\I$-space. If we fix $\mathbf{n}$, then $G^{(\bullet)}(\mathbf{n})$ is a simplicial space. Its geometric realization is the classifying space of the group $G(\mathbf{n})$, which we will denote by $B_{\nu}G(\mathbf{n})$. Any morphism $\varphi \colon \mathbf{m} \to \mathbf{n}$ in $\I$ induces a simplicial map $G^{(\bullet)}(\mathbf{m}) \to G^{(\bullet)}(\mathbf{n})$ and therefore a map on the corresponding classifying spaces: $B_{\nu}G(\mathbf{m}) \to B_{\nu}G(\mathbf{n})$. This way $\mathbf{n} \mapsto B_{\nu}G(\mathbf{n})$ becomes an $\I$-space, of which we can form the homotopy colimit $(B_{\nu}G)_{h\I}$. 

Alternatively, we can first form the homotopy colimit of the $r$-simplices to obtain the simplicial space $G^{(r)}_{h\I} = \hocolim_{\I} G^{(r)}$. Let $B_{\nu}(G_{h\I}) := \lvert G^{(\bullet)}_{h\I} \rvert$. As the notation already suggests there is not much of a difference between these two spaces. 

\begin{lemma} \label{lem:commutes}
Let $\mathcal{J}$ be a small (discrete) category, let $X$ be a simplicial $\mathcal{J}$-space, then $c \mapsto \lvert X^{(\bullet)}(c) \rvert$ is a $\mathcal{J}$-space, $r \mapsto \hocolim_{\mathcal{J}} X^{(r)}$ is a simplicial space and there is a homeomorphism
$
	\lvert \hocolim_{\mathcal{J}} X^{(\bullet)} \rvert \cong \hocolim_{\mathcal{J}} \lvert X^{(\bullet)} \rvert \ .
$
In particular, we have $(B_{\nu}G)_{h\I} \cong B_{\nu}(G_{h\I}) =: B_{\nu}G_{h\I}$. Given another category $\mathcal{J}'$, a functor $F \colon \mathcal{J} \to \mathcal{J}'$, a simplicial $\mathcal{J}'$-space $X'$ and a natural transformation $\kappa \colon X \Rightarrow X' \circ F$, the following diagram, in which the vertical maps are induced by $F$ and $\kappa$, commutes
\[
	\xymatrix{
		\lvert \hocolim_{\mathcal{J}} X^{(\bullet)} \rvert \ar[rr]^-{\cong} \ar[d] && \hocolim_{\mathcal{J}} \lvert X^{(\bullet)} \rvert \ar[d] \\
		\lvert \hocolim_{\mathcal{J}'} X'^{(\bullet)} \rvert \ar[rr]_-{\cong} && \hocolim_{\mathcal{J}'} \lvert X'^{(\bullet)} \rvert
	}
\] 
\end{lemma}

\begin{proof} \label{pf:commutes}
The first two statements are clear, since the face and degeneracy maps are maps of $\mathcal{J}$-spaces. In particular, the space $X^{(r)}$ of $r$-simplices is a $\mathcal{J}$-space. Let $\mathcal{C}_X^{(r)}$ be the transport category of $X^{(r)}$ and let $N_s \mathcal{C}_X^{(r)}$ be the $s$th space of the nerve of $\mathcal{C}_X^{(r)}$. This is a bisimplicial space and we have
\[
	\lvert \hocolim_{\mathcal{J}} X^{(\bullet)} \rvert = \lvert \lvert N_s \mathcal{C}_X^{(r)} \rvert_s \rvert_r \cong \lvert \lvert N_s \mathcal{C}_X^{(r)} \rvert_r \rvert_s \cong \hocolim_{\mathcal{J}} \lvert X^{\bullet} \rvert\ . 
\]
The functor $F$ in combination with the natural transformation yield a map of bisimplicial spaces $N_s\mathcal{C}_X^{(r)} \to N_s\mathcal{C}_{X'}^{(r)}$. So the last claim follows from the fact that the homeomorphism $\lvert \lvert N_s \mathcal{C}_X^{(r)} \rvert_s \rvert_r \cong \lvert \lvert N_s \mathcal{C}_X^{(r)} \rvert_r \rvert_s$ is natural in the bisimplicial space.
\end{proof}

\begin{lemma}\label{lem:Hocolim_is_BG} 
Let $G$ be a stable EH-$\I$-group, such that $G(\mathbf{n})$ has the homotopy type of a CW-complex for each $\mathbf{n}$. The inclusion map $G(\mathbf{1})^m \to G^{(m)}_{h\I}$ induces a homotopy equivalence for every $m \in \N$. In particular, 
$
	B_{\nu}G(\mathbf{1}) \to B_{\nu}G_{h\I}
$
is a homotopy equivalence.
\end{lemma}

\begin{proof} \label{pf:Hocolim_is_BG}
Since we assumed all maps $G(\mathbf{m}) \to G(\mathbf{n})$ to be homotopy equivalences for every $m \geq 1$, the Lemma follows from \cite[Lemma 2.1]{paper:Schlichtkrull} together with 
\cite[Proposition A.1 (ii)]{paper:SegalCatAndCoh}.
\end{proof}

Due to the Eckmann-Hilton condition the following diagram commutes
\[
	\xymatrix{
		G^{(r)}(\mathbf{m}) \times G^{(r)}(\mathbf{n}) \ar[r]^-{\mu^{(r)}_{m,n}} \ar[d]_{d_{i,m} \times d_{i,n} } & G^{(r)}(\mathbf{m} \sqcup \mathbf{n}) \ar[d]^-{d_{i,m \sqcup n}} \\
		G^{(r-1)}(\mathbf{m}) \times G^{(r-1)}(\mathbf{n}) \ar[r]_-{\mu^{(r-1)}_{m,n}} & G^{(r-1)}(\mathbf{m} \sqcup \mathbf{n})
	}
\]
Thus, $G^{(\bullet)}$ is a simplicial $\I$-monoid, which is commutative by Lemma \ref{lem:commutative}. This implies that $\mathbf{n} \mapsto B_{\nu}G(\mathbf{n})$ is a commutative $\I$-monoid. Let $B_{\nu}G^{(s)} \colon \I^s \to \mathcal{T}op_{\ast}$ be the $\I^s$-space given by $B_{\nu}G^{(s)}(\mathbf{n}_1, \dots, \mathbf{n}_s) = B_{\nu}G(\mathbf{n}_1) \times \dots \times B_{\nu}G(\mathbf{n}_s)$. As explained in \cite[section 5.2]{paper:Schlichtkrull} there is a $\Gamma$-space $\Gamma(B_{\nu}G)$ associated to $B_{\nu}G$ with
\[
	\Gamma(B_{\nu}G)(S) = \hocolim_{D(S)} B_{\nu}G^{(s)} \circ \pi_S\ ,
\] 
where $\pi_S \colon D(S) \to \I^{s}$ for $S = \{0, \dots, s\}$ is the canonical projection functor. All $G^{(r)}$ are commutative $\I$-monoids. Therefore we have analogous $\Gamma$-spaces $\Gamma(G^{(r)})$ and $\Gamma(G)$. Following \cite[Definition 1.3]{paper:SegalCatAndCoh}, we can deloop $\Gamma(G)$: Let $Y^{(k)}(S) = \Gamma(G)([k] \wedge S)$. This is a simplicial space and 
$
	B_{\mu}\Gamma(G)(S) = \lvert Y^{(\bullet)}(S) \rvert
$
is another $\Gamma$-space. 

In \cite{book:Geometric_App_of_Hom} Bousfield and Friedlander discussed two model category structures on $\Gamma$-spaces: a strict and a stable one. The strict homotopy category of very special $\Gamma$-spaces is equivalent to the stable homotopy category of connective spectra \cite[Theorem 5.1]{book:Geometric_App_of_Hom}. Instead of topological spaces, Bousfield and Friedlander considered (pointed) simplicial sets as a target category. The discussion in \cite[Appendix B]{paper:Schwede_Gamma} shows that this difference is not essential. Recall from \cite[Theorem 3.5]{book:Geometric_App_of_Hom} that a \emph{strict equivalence} between $\Gamma$-spaces $X$ and $Y$ is a natural transformation $f \colon X \to Y$, such that $f_S \colon X(S) \to Y(S)$ is a weak equivalence for every pointed set $S$ with $\lvert S \rvert \geq 2$. This agrees with the notion of strict equivalence given in \cite{paper:Schwede_Gamma}.

\begin{theorem}\label{thm:EHI_group}
Let $G$ be a convergent EH-$\I$-group, such that each $G(\mathbf{n})$ has the homotopy type of a CW-complex. Let $B_{\mu}G_{h\I}$ be the classifying space of the topological monoid $(G_{h\I},\mu)$. There is a strict equivalence of very special $\Gamma$-spaces $B_{\mu}\Gamma(G) \simeq \Gamma(B_{\nu}G)$ inducing a stable equivalence of the induced spectra in the stable homotopy category. 

In particular, $B_{\nu}G_{h\I} \simeq B_{\mu} G_{h\I}$. If $G$ is stable, then $B_{\mu}G_{h\I}$ is a classifying space for principal $G(\mathbf{1})$-bundles and $B_{\nu} G(\mathbf{1})$ is an infinite loop space.
\end{theorem}

\begin{proof} \label{pf:EHI_group}
Let $G^{(r,s)} \colon \I^s \to \mathcal{T}op_{\ast}$ be the simplicial $\I^s$-space given by
\[
	G^{(r,s)}(\mathbf{n}_1, \dots, \mathbf{n}_s) = G^{(r)}(\mathbf{n}_1) \times \dots \times G^{(r)}(\mathbf{n}_s)\ .
\]
Observe that $G^{(r,s)} \circ \pi_S$ is a simplicial $D(S)$-space with $\lvert G^{(\bullet,s)} \circ \pi_S \rvert = B_{\nu}G^{(s)} \circ \pi_S$, therefore Lemma \ref{lem:commutes} yields a homeomorphism $\Gamma(B_{\nu} G)(S) \cong \lvert \Gamma(G^{(\bullet)})(S)\rvert$. By the commutativity of the diagram in Lemma \ref{lem:commutes}, this is natural in $S$. Let $\lvert \Gamma(G^{(\bullet)}) \rvert (S) := \lvert \Gamma(G^{(\bullet)})(S)\rvert$. We obtain a levelwise homeomorphism of $\Gamma$-spaces $\Gamma(B_{\nu}G) \cong \lvert \Gamma(G^{(\bullet)}) \rvert$. In particular, this is a strict equivalence. The projection maps $\pi_k \colon G^{(r)}(\mathbf{n}) = G(\mathbf{n})^r \to G(\mathbf{n}) = G^{(1)}(\mathbf{n})$ induce 
\[
	\pi \colon \Gamma(G^{(r)})(S) \to \Gamma(G)(S) \times \dots \times \Gamma(G)(S) \ .
\]
For $S = S^0$ we have $\Gamma(G^{(r)})(S) = G^{(r)}_{h\I}$, $\Gamma(G)(S) = G_{h\I}$. $\pi$ fits into the commutative diagram
\[
	\xymatrix{
		G^{(r)}_{h\I} \ar[rr]^-{\pi} & & G_{h\I} \times \dots \times G_{h\I} \ \\
		\Tel(G^{(r)}) \ar[rr]^-{\widehat{\pi}} \ar[u]^-{\simeq} & & \Tel(G) \times \dots \times \Tel(G) \ar[u]_-{\simeq}
	}
\]
in which the homotopy equivalences follow from \cite[Lemma 2.1]{paper:Schlichtkrull}, since $G$ is convergent. Note that $\pi_k(\Tel(G)^{(r)}) = \varinjlim (\pi_k(G(\mathbf{n})) \times \dots \times \pi_k(G(\mathbf{n})))$, where the limit runs over the directed poset $\N_0$. Moreover, $\pi_k(\Tel(G)^r) = \varinjlim(\pi_k(G(\mathbf{n}_1)) \times \dots \times \pi_k(G(\mathbf{n}_r)))$, where the limit runs over the poset $\N_0^r$. The map $\widehat{\pi}$ is induced by the diagonal $\Delta_r \colon \N_0 \to \N_0^r$. Since the subset $\Delta_r(\N_0)$ is cofinal in $\N_0^r$, we obtain an isomorphism on all homotopy groups and therefore -- by our assumption on $G$ -- a homotopy equivalence.

Let now $S$ be arbitrary, let $s = \lvert S \rvert$. Let $G^{(r,s)}_{h\I} = \hocolim_{\I^s} G^{(r,s)}$. By \cite[Lemma 5.1]{paper:Schlichtkrull}, the natural map $\Gamma(G^{(r)})(S) \to G^{(r,s)}_{h\I}$ is an equivalence. From the above, we obtain that $\widetilde{\pi} \colon G^{(r,s)}_{h\I} \to G^{(1,s)}_{h\I} \times \dots \times G^{(1,s)}_{h\I}$ is an equivalence. The diagram
\[
	\xymatrix{
		\Gamma(G^{(r)})(S) \ar[r]^-{\pi} \ar[d]_-{\simeq} & \Gamma(G)(S) \times \dots \times \Gamma(G)(S) \ar[d]^-{\simeq} \\
		G^{(r,s)}_{h\I} \ar[r]^-{\widetilde{\pi}}_-{\simeq} & G^{(1,s)}_{h\I} \times \dots \times G^{(1,s)}_{h\I}
	}
\]
shows that $\pi$ is an equivalence for arbitrary $S$, which in turn implies that the induced map $\pi \colon B_{\mu}\Gamma(G^{(r)})(S) \to B_{\mu}\Gamma(G)(S) \times \dots \times B_{\mu}\Gamma(G)(S)$ is an equivalence as well. Now observe that $B_{\mu}\lvert\Gamma(G^{(\bullet)})\rvert(S) \cong \lvert B_{\mu}\Gamma(G^{(\bullet)}) \rvert(S)$ and $[k] \mapsto B_{\mu}\Gamma(G^{(k)})(S)$ is a simplicial space satisfying the properties of \cite[Proposition 1.5]{paper:SegalCatAndCoh}. In particular, $\pi_0(B_{\mu}\Gamma(G)(S)) \cong \pi_0((B_{\mu}G_{h\I})^s)$ is trivial. Therefore $B_{\mu}\Gamma(G)(S) \to \Omega\lvert B_{\mu}\Gamma(G^{(\bullet)}) \rvert(S) \cong \Omega B_{\mu}\lvert\Gamma(G^{(\bullet)})\rvert(S)$ is a homotopy equivalence, which is natural in $S$. Altogether we obtain a sequence of strict equivalences
\[
	\Gamma(B_{\nu}G) \simeq \Omega B_{\mu} \Gamma(B_{\nu}G) \simeq \Omega B_{\mu} \lvert \Gamma(G^{(\bullet)})\rvert \simeq B_{\mu} \Gamma(G)\ .
\]
We used that $B_{\nu}G$ is connected (and therefore $\Gamma(B_{\nu}G)$ is a very special $\Gamma$-space) in the first step. $S = S^0$ together with Lemma \ref{lem:Hocolim_is_BG} yields $B_{\nu}G(\mathbf{1}) \simeq  B_{\nu}G_{h\I} \simeq B_{\mu}G_{h\I}$ in case $G$ is stable.
\end{proof}

\subsection{Actions of Eckmann-Hilton $\I$-groups on spectra}
As was alluded to in the introduction, the compatibility diagram of an EH-$\I$-group $G$ from Definition \ref{def:EHI_group} enables us to talk about the action of $G$ on a commutative symmetric ring spectrum $R$. 

\begin{definition} \label{def:EHI_action}
Let $G$ be an EH-$\I$-group and let $R$ be a commutative symmetric ring spectrum. Then $G$ is said to \emph{act on} $R$, if $G(\mathbf{n})$ acts on $R_n$ via $\kappa_n \colon G(\mathbf{n}) \times R_n \to R_n$ such that 
\begin{enumerate}[(i)]
	\item \label{it:equi} $\kappa_n$ preserves the basepoint of $R_n$ and is $\Sigma_n$-equivariant, where $\Sigma_n$ acts on $G(\mathbf{n}) \times R_n$ diagonally and $G(\mathbf{0})$ acts trivially on $R_0$,
	\item \label{it:mult} the action is compatible with the multiplication in the sense that the following diagram commutes
	\[
		\xymatrix{
			G(\mathbf{m}) \times R_m \times G(\mathbf{n}) \times R_n \ar[rr]^-{\kappa_m \times \kappa_n} \ar[d]_-{(\mu^G_{m,n} \times \mu^R_{m,n}) \circ \tau }& & R_m \wedge R_n \ar[d]^-{\mu^R_{m,n}} \\
			G(\mathbf{m} \sqcup \mathbf{n}) \times R_{m+n} \ar[rr]_-{\kappa_{m+n}} & & R_{m+n}
		}
	\]
	where $\tau$ denotes the map that switches the two inner factors. 
	\item \label{it:stab} the action is compatible with stabilization in the sense that the following diagram commutes for $l,m,n \in \N_0$ with $l+m = n$, $\sigma_{l,m} = \mu_{l,m} \circ (\eta_l \wedge \id{R_m})$ and the order preserving inclusion $\iota_{m,n} \colon \mathbf{m} \to \mathbf{n}$ onto the last $m$ elements
	\[
		\xymatrix{
			G(\mathbf{m}) \times (S^l \wedge R_m) \ar[rr]^-{\id{S^l} \wedge \kappa_m} \ar[d]_-{\iota_{m,n*} \times \sigma_{l,m}} & & S^l \wedge R_m \ar[d]^-{\sigma_{l,m}} \\
			G(\mathbf{n}) \times R_n \ar[rr]_-{\kappa_n} & & R_n
		}
	\]
\end{enumerate}
\end{definition}

\begin{theorem}\label{thm:EHI_action}
An action of an EH-$\I$-group $G$ on a comm.\ symmetric ring spectrum $R$ defines a map of $\I$-mo\-noids $G \to \Omega^{\infty}(R)$, which sends $g \in G(\mathbf{n})$ to $g \cdot \eta_n := \kappa_n(g, \eta_n) \in \Omega^{\infty}(R)(\mathbf{n}) = \Omega^nR_n$. If $G$ has compatible inverses, this factors over a morphism $G \to \units{R}$ of commutative $\I$-monoids, which deloops to a map
\(
	B_\mu \Gamma(G) \to B_{\mu}\Gamma(\units{R})
\).	In particular, we obtain $B_{\mu}G_{h\I} \to BGL_1(R)$.
\end{theorem}

\begin{proof}\label{pf:EHI_action}
To see that $G \to \Omega^{\infty}(R)$	really defines a natural transformation, observe that each morphism $\alpha \colon \mathbf{m} \to \mathbf{n}$ factors as $\alpha = \bar{\alpha} \circ \iota_{m,n}$ with $\bar{\alpha} \in \Sigma_n$ as explained in the paragraph after Definition \ref{def:ISpace_IMonoid}. Now note that
\begin{align*}
	\alpha_*(g \cdot \eta_m) & = \bar{\alpha} \circ \mu^R_{l,m}(\eta_l \wedge (g \cdot \eta_m)) \circ \bar{\alpha}^{-1} = \bar{\alpha} \circ \iota_{m,n*}(g) \cdot \mu^R_{l,m}(\eta_l \wedge \eta_m) \circ \bar{\alpha}^{-1} \\
	& = \alpha_*(g) \cdot \bar{\alpha} \circ \eta_n \circ \bar{\alpha}^{-1} = \alpha_*(g) \cdot \eta_n
\end{align*}
where we used (\ref{it:equi}) and (\ref{it:stab}) of Definition \ref{def:EHI_action}. That $G \to \Omega^{\infty}(R)$ is a morphism of $\I$-monoids is a consequence of (\ref{it:mult}). Indeed, for $g \in G(\mathbf{m})$, $h \in G(\mathbf{n})$ 
\[
	\mu^R_{m,n}(g \cdot \eta_m, h \cdot \eta_n) = \mu^G_{m,n}(g,h) \cdot \mu^R_{m,n}(\eta_m, \eta_n) = \mu^G_{m,n}(g,h) \cdot \eta_{m+n}\ .
\]
If $G$ has compatible inverses, then $g^{-1}$ provides a stable inverse of $g \in G(\mathbf{n})$ by Lemma~\ref{lem:h_inverse}. Forming homotopy colimits, we obtain a morphism of topological monoids $G_{h\I} \to GL_1(R)$, which deloops.
\end{proof}

\subsection{Eckmann-Hilton $\I$-groups and permutative categories}
Given a permutative category $(\mathcal{C}, \otimes)$ and an object $x \in {\rm obj}(\mathcal{C})$, there is a canonical commutative $\I$-monoid $E_x$ associated to it: Let $E_x(\mathbf{n}) = {\rm End}_{\mathcal{C}}(x^{\otimes n})$ (with $x^{\otimes 0} = 1_{\mathcal{C}}$), let $\bar{\alpha} \in \Sigma_n$ be the permutation associated to a morphism $\alpha \colon \mathbf{m} \to \mathbf{n}$ as above and define $\alpha_* \colon E_x(\mathbf{m}) \to E_x(\mathbf{n})$ by sending an endomorphism $f$ to $\bar{\alpha} \circ (\id{x^{\otimes n-m}} \otimes f) \circ \bar{\alpha}^{-1}$, where the permutation group $\Sigma_n$ acts on $x^{\otimes n}$ using the symmetry of $\mathcal{C}$. The monoid structure of $E_x$ is given by 
\[
	\mu_E \colon E_x(\mathbf{m}) \times E_x(\mathbf{n}) \to E_x(\mathbf{m} \sqcup \mathbf{n}) \quad ; \quad (f,g) \mapsto f \otimes g\ .
\]
Let $A_x(\mathbf{n}) = {\rm Aut}_{\mathcal{C}}(x^{\otimes n})$ together with the analogous structures as described above. This is an Eckmann-Hilton $\I$-group. Let $\mathcal{C}_x$ be the full permutative subcategory of $\mathcal{C}$ containing the objects $x^{\otimes n}$, $n \in \N_0$.

\begin{definition}\label{def:stabilization}
A strict symmetric monoidal functor $\theta \colon \I \to \mathcal{C}_x$ will be called a \emph{stabilization of $x$}, if $\theta(\mathbf{1}) = x$ and for each morphism $\alpha \colon \mathbf{m} \to \mathbf{n}$ in $\I$ and each $f \in E_x(\mathbf{m})$ the following diagram commutes:
\[
	\xymatrix{
	\theta(\mathbf{m}) \ar[r]^{f} \ar[d]_{\theta(\alpha)} & \theta(\mathbf{m}) \ar[d]^{\theta(\alpha)}\\
	\theta(\mathbf{n}) \ar[r]_{\alpha_*(f)} & \theta(\mathbf{n})
	}
\]
\end{definition}

\begin{lemma} \label{lem:bij_stabi}
Let $\iota_{\mathbf{1}} \colon \mathbf{0} \to \mathbf{1}$ be the unique morphism in $\I$. The map that associates to a stabilization $\theta$ the morphism $\theta(\iota_{\mathbf{1}}) \in \hom_{\mathcal{C}}(1_{\mathcal{C}}, x)$ yields a bijection between stabilizations and elements in $\hom_{\mathcal{C}}(1_{\mathcal{C}},x)$. 
\end{lemma}

\begin{proof} \label{pf:bij_stabi}
Note that a stabilization is completely fixed by knowing $\theta(\iota_{\mathbf{1}})$, therefore the map is injective. Let $\varphi \in \hom_{\mathcal{C}}(1_{\mathcal{C}},x)$. Define $\theta(\mathbf{n}) = x^{\otimes n}$ on objects. Let $\alpha \colon \mathbf{m} \to \mathbf{n}$ be a morphism in $\I$ and let $\alpha = \bar{\alpha} \circ \iota_{m,n}$ be the factorization as explained after Definition~\ref{def:ISpace_IMonoid}. Define $\theta(\iota_{m,n}) = \varphi^{\otimes (n-m)} \otimes \id{x^{\otimes m}} \colon x^{\otimes m} = 1_{\mathcal{C}}^{\otimes (n-m)} \otimes x^{\otimes m} \to x^{\otimes n}$ and let $\theta(\bar{\alpha})$ be the permutation of the tensor factors. The decomposition $\alpha = \bar{\alpha} \circ \iota_{m,n}$ is \emph{not} functorial, due to the fact that $\overline{\beta \circ \alpha}$ and $\bar{\beta} \circ (\id{} \sqcup \bar{\alpha})$ differ by a permutation. Nevertheless $\theta(\alpha) = \theta(\bar{\alpha}) \circ \theta(\iota_{m,n})$ turns out to be functorial due to the permutation invariance of $\varphi^{\otimes k}$. It is straightforward to check that this is also strict symmetric monoidal. Let $f \colon x^{\otimes m} \to x^{\otimes m}$ be in $E_x(\mathbf{m})$. We have
\[
	\theta(\alpha) \circ f = \bar{\alpha} \circ (\varphi^{\otimes (n-m)} \otimes \id{x^{\otimes m}}) \circ f = \bar{\alpha} \circ (\id{x^{\otimes n-m}} \otimes f) \circ \bar{\alpha}^{-1} \circ \theta(\alpha) = \alpha_*(f) \circ \theta(\alpha)\ .
\]
This shows that the map is also surjective.
\end{proof}

As sketched in \cite{paper:SegalCatAndCoh} there is a $\Gamma$-space $\Gamma(\mathcal{C})$ associated to a $\Gamma$-category $\mathcal{A}_\mathcal{C}$, which is constructed as follows: Let $S$ be a finite pointed set and denote by $\bar{S}$ the complement of the basepoint of $S$, then the objects of $\mathcal{A}_\mathcal{C}(S)$ are families $\{ x_U \in {\rm obj}(\mathcal{C})\ |\ U \subset \bar{S}\}$ together with isomorphisms $\alpha_{U,V} \colon x_U \otimes x_V \to x_{U \cup V}$, whenever $U \cap V = \emptyset$, compatible with the symmetry of $\mathcal{C}$ and such that $x_{\emptyset} = 1_{\mathcal{C}}$ and $x_U \otimes x_{\emptyset} \to x_U$ is the identity. The morphisms of $\mathcal{A}_\mathcal{C}(S)$ are families of morphisms $\beta_{U,U'} \colon x_U \to x_{U'}$ in $\mathcal{C}$ such that 
\[
	\xymatrix{
		x_U \otimes x_V \ar[rr]^-{\alpha_{U,V}} \ar[d]_-{\beta_{U,U'} \otimes \beta_{V,V'}} && x_{U \cup V} \ar[d]^-{\beta_{U \cup V, U' \cup V'}} \\
		x_{U'} \otimes x_{V'} \ar[rr]_-{\alpha_{U',V'}} && x_{U' \cup V'}
	}
\]
commutes. We define $\Gamma(\mathcal{C})(S) = \lvert \mathcal{A}_{\mathcal{C}}(S) \rvert$. Let $\mathcal{C}_x$ be the full (permutative) subcategory of $\mathcal{C}$ containing the objects $x^{\otimes n}$ for $n \in \N_0$.

In the next lemma we will introduce a technical tool from simplicial homotopy theory: The nerve of a topological category $\mathcal{D}$ is the simplicial space $N_n\mathcal{D} = {\rm Fun}([n],\mathcal{D})$, where $[n]$ is the category of the directed poset $\{0, \dots, n\}$. We define the \emph{double nerve} of $\mathcal{D}$ as the  \emph{bi}simplicial space $N_{\bullet}N_{\bullet}\mathcal{D}$ with $N_m N_n \mathcal{D} = {\rm Fun}([m] \times [n], \mathcal{D})$. It consists of an $m \times n$-array of commuting squares in $\mathcal{D}$. The diagonal functors ${\rm diag}_n \colon [n] \to [n] \times [n]$ induce a simplicial map $\varphi_n \colon N_n N_n \mathcal{D} \to N_n \mathcal{D}$. 

\begin{lemma} \label{lem:bisimp_vs_simp}
The simplicial map $\varphi_{\bullet}$ induces a homotopy equivalence $\lvert N_{\bullet}N_{\bullet}\mathcal{D} \rvert \to \lvert N_{\bullet}\mathcal{D} \rvert$.
\end{lemma}
	
\begin{proof} \label{pf:bisimp_vs_simp}
To construct the homotopy inverse, let $\max_n \colon [n] \times [n] \to [n]$ be given by $\max_n(k,\ell) = \max\{k, \ell \}$, which completely determines its value on morphisms. We have $\max_n \circ\, {\rm diag}_n = \id{[n]}$. Let $(k,\ell) \in [n] \times [n]$. There is a unique morphism $\kappa_{(k,\ell)} \colon (k, \ell) \to (\max\{k,\ell\}, \max\{k,\ell\})$. Therefore there is a natural transformation $\kappa \colon \id{[n] \times [n]} \to \,{\rm diag}_n \circ \max_n$. From $\kappa$, we can construct a functor $h \colon [n] \times [n] \times [1] \to [n] \times [n]$, which induces 
\[
	H \colon N_{n}N_n\mathcal{D} \times \hom([n] \times [n], [1]) \to N_nN_n\mathcal{D} \ ; \  (F, f) \mapsto F \circ h \circ (\id{[n] \times [n]} \times f) \circ {\rm diag}_{[n] \times [n]}\ .
\]
After geometric realization, $H$ yields a homotopy inverse of ${\rm diag}_n \circ \max_n$.
\end{proof}

Let $S_{\otimes} \colon \mathcal{C}_x \to \mathcal{C}_x$ be the functor $x^{\otimes n} \mapsto x^{\otimes (n+1)}$ and $f \mapsto f \otimes \id{x}$. Observe that 

\noindent $\lvert \Tel(N_k \mathcal{C}_x; N_kS_{\otimes}) \rvert_k \cong \Tel(\lvert N_{\bullet}\mathcal{C}_x \rvert; \lvert N_{\bullet}S_{\otimes} \rvert) = \Tel(\lvert N_{\bullet} \mathcal{C}_x \rvert)$ by Lemma~\ref{lem:commutes} and that $\theta$ provides a natural transformation $\id{\mathcal{C}_x} \to S_{\otimes}$. Therefore we obtain a map $\Tel(\lvert N_{\bullet}\mathcal{C}_x \rvert) \to \lvert N_{\bullet}N_{\bullet}\mathcal{C}_x \rvert$.

\begin{lemma} \label{lem:telescope}
The map $\Tel(\lvert N_{\bullet}\mathcal{C}_x \rvert) \to \lvert N_{\bullet}N_{\bullet}\mathcal{C}_x \rvert$ constructed in the last paragraph is a homotopy equivalence.
\end{lemma}

\begin{proof} \label{pf:telescope}
$\theta$ yields a natural transformation $\id{\mathcal{C}_x} \to S_{\otimes}$. Therefore $\lvert N_{\bullet}S_{\otimes} \rvert$ is homotopic to the identity and the map $\lvert N_{\bullet}\mathcal{C}_x \rvert \to \Tel(\lvert N_{\bullet}\mathcal{C}_x \rvert)$ onto the $0$-skeleton is a homotopy equivalence. Likewise, $\lvert N_{\bullet} \mathcal{C}_x \rvert \to \lvert N_{\bullet} N_{\bullet} \mathcal{C}_x \rvert$ induced by the simplicial map $N_{\ell}\,\mathcal{C}_x \to N_{\ell}N_{\ell}\,\mathcal{C}_x$ sending a diagram to the corresponding square that has $\ell$ copies of the diagram in its rows and only identities as vertical maps is a homotopy equivalence by Lemma \ref{lem:bisimp_vs_simp} since composition with $\varphi_{\bullet}$ yields the identity on $N_{\ell}\,\mathcal{C}_x$. The statement now follows from the commutative triangle
\[
	\xymatrix{
	\Tel(\lvert N_{\bullet}\mathcal{C}_x \rvert) \ar[rr] & & \lvert N_{\bullet} N_{\bullet} \mathcal{C}_x \rvert \\
	& \lvert N_{\bullet} \mathcal{C}_x \rvert \ar[ur]_-{\simeq} \ar[ul]^-{\simeq} 
	}
\]
which finishes the proof.
\end{proof}

\begin{theorem} \label{thm:Schlichtkrull_vs_Segal}
Let $x \in {\rm obj}(\mathcal{C})$ and let $\theta \colon \I \to \mathcal{C}_x$ be a stabilization of $x$. Let $G = A_x$. There is a map of $\Gamma$-spaces 
\[
	\Phi_\theta \colon \Gamma(B_{\nu}G) \to \Gamma(\mathcal{C}_x)\ .	
\]
If $G$ is convergent and there exists an unbounded non-decreasing sequence of natural numbers $\lambda_n$, such that $B_{\nu}G(\mathbf{m}) \to \lvert N_{\bullet} \mathcal{C}_x \lvert$ is $\lambda_n$-connected for all $m > n$, then $\Phi_{\theta}$ is an equivalence as well.
\end{theorem}

\begin{proof} \label{pf:Schlichtkrull_vs_Segal}
Lemma~\ref{lem:commutes} yields a homeomorphism $\Gamma(B_{\nu}G)(S) \cong \lvert \Gamma(G^{(\bullet)})(S) \rvert$ with 

\noindent $\Gamma(G^{(r)})(S) = \hocolim_{D(S)} G^{(r,s)} \circ \pi_S$. The last space is the geometric realization of the transport category $\mathcal{T}^{(r)}_G(S)$ with object space $\coprod_{d \in {\rm obj}(D(S))} G^{(r,s)}(\pi_S(d))$ and morphism space 
\[
	{\rm mor}(\mathcal{T}^{(r)}_G(S)) = \coprod_{d,d' \in {\rm obj}(D(S))} G^{(r,s)}(\pi_S(d)) \times \hom_{D(S)}(d,d')\ .	
\]
Given a diagram $d \in {\rm obj}(D(S))$, we obtain $\mathbf{n}_U = d(U) \in {\rm obj}(\I)$ for every $U \subset \bar{S}$ and define $\mathbf{n}_i = d(\{i\})$. Let $(g_1, \dots g_s) \in G^{(r)}(\mathbf{n}_1) \times \dots \times G^{(r)}(\mathbf{n}_s)$. We may interpret $g_i$ as an $r$-tupel $(g_{j,i})_{j \in \{1, \dots, r\}}$ of automorphisms $g_{j,i} \in G(\mathbf{n}_i) = {\rm Aut}_{\mathcal{C}}(\theta(\mathbf{n}_i))$. If $U,V \subset \bar{S}$ are two subsets with $U \cap V = \emptyset$, then $d$ yields two morphisms $\mathbf{n}_U \to \mathbf{n}_{U \cup V}$ and $\mathbf{n}_V \to \mathbf{n}_{U \cup V}$, which form an isomorphism $\iota_{U,V} \colon \mathbf{n}_U \sqcup \mathbf{n}_V \to \mathbf{n}_{U \cup V}$ in $\I$. Let $\alpha_{U,V} = \theta(\iota_{U,V}) \colon \theta(\mathbf{n}_U) \otimes \theta(\mathbf{n}_V) \to \theta(\mathbf{n}_{U \cup V})$. Let $\iota^i_U \colon \mathbf{n}_i \to \mathbf{n}_U$ be induced by the inclusion $\{i\} \subset U$ and define $g_{j,U} = \prod_{i \in U} \iota^i_{U*}(g_{j,i})$. This does not depend on the order, in which the factors are multiplied. The elements $g_{j,U}$ fit into a commutative diagram
\[
\xymatrix{
	\theta(\mathbf{n}_U) \otimes \theta(\mathbf{n}_V) \ar[rr]^-{\alpha_{U,V}} \ar[d]_-{g_{j,U} \otimes g_{j,V}} && \theta(\mathbf{n}_{U \cup V}) \ar[d]^-{g_{j,U \cup V}} \\
	\theta(\mathbf{n}_U) \otimes \theta(\mathbf{n}_V) \ar[rr]_-{\alpha_{U,V}} && \theta(\mathbf{n}_{U \cup V})
}
\]
Thus, we can interpret the families $g_{j,U}$ for $j \in \{1, \dots, r\}$ as an $r$-chain of automorphisms of the object $(\theta(\mathbf{n}_U), \alpha_{U,V})$ in $\mathcal{A}_{\mathcal{C}}(S)$.

Let $d'$ be another diagram, let $\mathbf{m}_U = d'(U)$, $\iota'_{U,V} \colon \mathbf{m}_U \sqcup \mathbf{m}_V \to \mathbf{m}_{U \cup V}$ and let $\beta_{U,V} = \theta(\iota'_{U,V})$. A natural transformation $d \to d'$ consists of morphisms $\varphi_U \colon \mathbf{n}_U \to \mathbf{m}_U$ for every $U \in \mathcal{P}(\bar{S})$. Let $f_U = \theta(\varphi_U)$, then the following diagrams commute:
\[
	\xymatrix{
		\theta(\mathbf{n}_U) \ar[rr]^-{g_{j,U}} \ar[d]_-{f_U} && \theta(\mathbf{n}_U) \ar[d]^-{f_U} \\
		\theta(\mathbf{m}_U) \ar[rr]_-{\varphi_{U*}(g_{j,U})} && \theta(\mathbf{m}_U)
	} \qquad 
	\xymatrix{
	\theta(\mathbf{n}_U) \otimes \theta(\mathbf{n}_V) \ar[rr]^-{\alpha_{U,V}} \ar[d]_-{f_U \otimes f_V} && \theta(\mathbf{n}_{U \cup V}) \ar[d]^-{f_{U \cup V}}\\
	\theta(\mathbf{m}_U) \otimes \theta(\mathbf{m}_V) \ar[rr]_-{\beta_{U,V}} && \theta(\mathbf{m}_{U \cup V})	
	}
\]
the first since $\theta$ is a stabilization, the second by naturality of the transformation. This is compatible with respect to composition of natural transformations of diagrams. Altogether we have constructed a bisimplicial map $N_{\ell}\mathcal{T}^{(r)}_G(S) \to N_{\ell}N_r\mathcal{A}_{\mathcal{C}}(S)$. Combining this with $N_{\ell}N_{\ell}\mathcal{A}_{\mathcal{C}}(S) \to N_{\ell}\mathcal{A}_{\mathcal{C}}(S)$ from Lemma \ref{lem:bisimp_vs_simp}, we obtain $\Gamma(B_{\nu}G)(S) \to \Gamma(\mathcal{C})(S)$ after geometric realization.

It remains to be proven that this map is functorial with respect to morphisms $\kappa \colon S \to T$ in $\Gamma^{op}$. Recall that 
\[
	N_{\ell}\mathcal{T}_G^{(r)}(S) = \coprod_{d_1, \dots, d_{\ell} \in {\rm obj}(D(S))} \prod_{s \in S} G^{(r)}(d_1(\{s\})) \times \hom_{D(S)}(d_1,d_2) \times \dots \times \hom_{D(S)}(d_{\ell-1},d_{\ell})
\]
Let $\kappa_* \colon N_{\ell}\mathcal{T}_G^{(r)}(S) \to N_{\ell}\mathcal{T}_G^{(r)}(T)$ be the induced map as defined in \cite[section~5.2]{paper:Schlichtkrull}. We have 
\[
	\kappa_*\left((g_{j,i})_{i \in S, j \in \{0,\dots,r\}}, \varphi^1, \dots, \varphi^{\ell-1}\right) = \left((h_{j,k})_{k \in T, j \in \{0,\dots,r\}}, \kappa_*\varphi^1, \dots, \kappa_*\varphi^{\ell-1}\right)
\]
where  $h_{j,k} = \prod_{i \in \kappa^{-1}(k)} \iota^i_{\kappa^{-1}(k)*}(g_{j,i})$. If the left hand side lies in the component $(d_1, \dots, d_{\ell})$, then the right hand side is in $(\kappa_*d_1, \dots, \kappa_*d_{\ell})$ with $(\kappa_*d_m)(V) = d_m(\kappa^{-1}(V))$ for $V \subset \bar{T}$. Likewise $(\kappa_*\varphi^m)_V = \varphi^m_{\kappa^{-1}(V)}$. The functor $\kappa_* \colon \mathcal{A}_{\mathcal{C}}(S) \to \mathcal{A}_{\mathcal{C}}(T)$ sends the object $(x_U, \alpha_{U,V})_{U,V \subset \bar{S}}$ to $(x_{\kappa^{-1}(\tilde{U})}, \alpha_{\kappa^{-1}(\tilde{U}), \kappa^{-1}(\tilde{V})})_{\tilde{U}, \tilde{V} \subset \bar{T}}$ and is defined analogously on morphisms. Observe that the composition $N_{\ell}\mathcal{T}_G^{(r)}(S) \to N_{\ell}\mathcal{T}_G^{(r)}(T) \to N_{\ell}N_r\mathcal{A}_{\mathcal{C}}(T)$ maps $(g_{j,i})_{i \in S, j\in \{0,\dots,r\}}$ to the $r$-chain of automorphisms given by $h_{j,V} \colon \theta(\kappa_*d(V)) \to \theta(\kappa_*d(V)) $ with
\[
h_{j,V} = \prod_{k \in V} \prod_{i \in \kappa^{-1}(k)} (\kappa_*d)(\iota^k_V) \circ d(\iota^{i}_{\kappa^{-1}(k)}) (g_{j,i}) = \prod_{i \in \kappa^{-1}(V)} d(\iota^i_{\kappa^{-1}(V)})(g_{j,i}) = g_{j,\kappa^{-1}(V)}
\]
for $V \subset \bar{T}$. The transformations $\varphi^m$ are mapped to $\theta(\varphi^m_{\kappa^{-1}(V)}) = f^m_{\kappa^{-1}(V)}$. This implies the commutativity of 
\[
	\xymatrix{
	N_{\ell}\mathcal{T}_G^{(r)}(S) \ar[rr]^{\kappa_*} \ar[d]_{\Phi_{\theta}} & & N_{\ell}\mathcal{T}_G^{(r)}(T) \ar[d]^{\Phi_{\theta}}\\
	N_{\ell}N_r\mathcal{A}_{\mathcal{C}}(S) \ar[rr]^{\kappa_*} & & N_{\ell}N_r\mathcal{A}_{\mathcal{C}}(T)
	}
\]
and therefore functoriality after geometric realization.

To see that $\Phi_{\theta}$ is an equivalence for each $S$ it suffices to check that $B_{\nu}G_{h\I} \to \lvert N_{\bullet} \mathcal{C} \rvert$ is a homotopy equivalence due to the following commutative diagram
\[
	\xymatrix{
		\Gamma(B_{\nu}G)(S) \ar[r] \ar[d]_-{\simeq} & \Gamma(\mathcal{C})(S) \ar[d]^-{\simeq}\\
	 	\left(B_{\nu}G_{h\I}\right)^s \ar[r] & \lvert N_{\bullet} \mathcal{C} \rvert^s
	}
\]
Let $\Tel(\lvert N_{\bullet} \mathcal{C}_x \rvert)$ be the telescope of $\lvert N_{\bullet}S_{\otimes} \rvert$ as defined above. The conditions on the maps $B_{\nu}G(\mathbf{n}) \to \lvert N_{\bullet} \mathcal{C}_x \rvert$ ensure that $\Tel(B_{\nu}G) \to \Tel(\lvert N_{\bullet} \mathcal{C}_x \rvert)$ is a homotopy equivalence. Now we have the following commutative diagram
\[
	\xymatrix{
	B_{\nu}G_{h\I} \ar[r] & \lvert N_{\bullet} N_{\bullet} \mathcal{C}_x \rvert \ar[r]^{\simeq} & \lvert N_{\bullet}\mathcal{C}_x \rvert \\
	\Tel(B_{\nu}G) \ar[r]^-{\simeq} \ar[u]^-{\simeq} & \Tel(\lvert N_{\bullet}\mathcal{C}_x \rvert) \ar[u]_-{\simeq}
	}
\]
in which the upper horizontal map is the one we are looking for. This finishes the proof. 
\end{proof}

\section{Strongly self-absorbing $C^*$-algebras and $gl_1(KU^A)$}
A $C^*$-algebra $A$ is called \emph{strongly self-absorbing} if it is separable, unital and there exists a $*$-isomorphism $\psi \colon A \to A\otimes A$ such that $\psi$ is approximately unitarily equivalent to the map $l \colon A \to A\otimes A$, $l(a)=a\otimes 1_A$ \cite{paper:TomsWinter}. This means that there is a sequence of unitaries $(u_n)$ in $A$ such that $\|u_n \varphi(a)u_n^*-l(a)\|\to 0$ as $n\to \infty$ for all $a\in A$. In fact, it is a consequence of \cite[Theorem~2.2]{paper:DadarlatKK1} and \cite{paper:WinterZStable} that $\psi$, $l$ and $r \colon A \to A \otimes A$ with $r(a) = 1_A \otimes a$ are homotopy equivalent and in fact the group
$\Aut{A}$ is contractible \cite{paper:DadarlatP1}. The inverse isomorphism $\psi^{-1}$ equips $K_{\ast}(C(X) \otimes A)$ with a ring structure induced by the tensor product. By homotopy invariance of $K$-theory, the $K_0$-class of the constant map on $X$ with value $1 \otimes e$ for a rank $1$-projection $e \in \K$ is the unit of this ring structure. To summarize: Given a separable, unital, strongly self-absorbing $C^*$-algebra $A$, the functor $X \mapsto K_*(C(X) \otimes A)$ is a multiplicative cohomology theory on finite CW-complexes. 

\subsection{A commutative symmetric ring spectrum representing $K$-theory}
A $C^*$-alge\-bra $B$ is \emph{graded}, if it comes equipped $\Z/2\Z$-action, i.e.\ a $*$-automorphism $\alpha \colon B \to B$, such that $\alpha^2 = \id{B}$. A \emph{graded homomorphism} $\varphi \colon (B,\alpha) \to (B',\alpha')$ has to satisfy $\varphi \circ \alpha = \alpha' \circ \varphi$. The algebraic tensor product $B \odot B'$ can be equipped with the multiplication and $*$-operation
\[
	(a\, \widehat{\otimes} \,b)(a'\, \widehat{\otimes}\, b') = (-1)^{\partial b \cdot \partial a' } (aa' \,\widehat{\otimes}\, bb') \quad \text{and} \quad (a\, \widehat{\otimes}\, b)^* = (-1)^{\partial a \cdot \partial b} (a^* \,\widehat{\otimes} \,b^*)
\]
where $a,a' \in B$ and $b,b' \in B'$ are homogeneous elements and $\partial a$ denotes the degree of $a$. It is graded via $\partial (a \,\widehat{\otimes} \,b) = \partial a + \partial b$ modulo $2$. The (minimal) graded tensor product $B\,\widehat{\otimes}\,B'$ is the completion of $B \odot B'$ with respect to the tensor product of faithful representations of $B$ and $B'$ on graded Hilbert spaces. For details we refer the reader to \cite[section 14.4]{book:Blackadar}. 

We define $\widehat{S} = C_0(\R)$ with the grading by even and odd functions. The Clifford algebra $\C\ell_1$ will be spanned by the even element $1$ and the odd element $c$ with $c^2 = 1$. The algebra $\widehat{\K}$ will denote the graded compact operators on a graded Hilbert space $H = H_0 \oplus H_1$ with grading ${\rm Ad}_u$, $u = \left(\begin{smallmatrix} 1 & 0 \\ 0 & -1 \end{smallmatrix}\right)$, whereas we will use $\K$ for the \emph{trivially} graded compact operators. If we take tensor products between a graded $C^*$-algebra and a trivially graded one, e.g.\ $\C\ell_1 \widehat{\otimes} \K$, we will write $\otimes$ instead of $\widehat{\otimes}$. It is a consequence of \cite[Corollary 14.5.5]{book:Blackadar} that 
\(
	(\C\ell_1 \otimes \K) \,\widehat{\otimes}\,(\C\ell_1 \otimes \K) \cong M_2(\C) \otimes \K \cong \widehat{\K}
\),
where $M_2(\C)$ is graded by ${\rm Ad}_u$ with $u = \left(\begin{smallmatrix} 1 & 0 \\ 0 & -1 \end{smallmatrix}\right)$. 

As is explained in \cite[section 3.2]{paper:FunctorialKSpectrum}, \cite[section 4]{paper:Joachim}, the algebra $\widehat{S}$ carries a counital, coassociative and cocommutative coalgebra structure. This arises from a $1:1$-correspondence between essential graded $*$-homomorphisms $\widehat{S} \to A$ and odd, self-adjoint, regular unbounded multipliers of $A$ \cite[Proposition 3.1]{paper:Trout}. Let $X$ be the multiplier corresponding to the identity map on $\widehat{S}$, then the comultiplication $\Delta \colon \widehat{S} \to \widehat{S}\,\otimes \widehat{S}$ is given by $1\,\widehat{\otimes}\,X + X\,\widehat{\otimes}\,1$, whereas the counit $\epsilon \colon \widehat{S} \to \C$ corresponds to $0 \in \C$, i.e.\ it maps $f \mapsto f(0)$. 

\begin{definition} \label{def:KUA_spectrum}
Let $A$ be a separable, unital, strongly self-absorbing and trivially graded $C^*$-algebra. Let $KU^A_{\bullet}$ be the following sequence of spaces
\[
	KU^A_n = \hom_{\rm gr}(\widehat{S}, (\C\ell_1 \otimes A \otimes \K)^{\widehat{\otimes} n})\ ,
\]
where the graded homomorphisms are equipped with the point-norm topology.
\end{definition}

$KU^A_n$ is pointed by the $0$-homomorphism and carries a basepoint preserving $\Sigma_n$-action by permuting the factors of the graded tensor product (this involves signs!). We set $B^{\widehat{\otimes} 0} = \C$ and observe that $KU^A_0$ is the two-point space consisting of the $0$-homomorphism and the evaluation at $0$, which is the counit of the coalgebra structure on $\widehat{S}$.

Let $\mu_{m,n}$ be the following family of maps 
\[
	\mu_{m,n} \colon KU^A_m \wedge KU^A_n \to KU^A_{m+n} \quad ; \quad \varphi \wedge \psi \mapsto (\varphi \,\widehat{\otimes}\, \psi) \circ \Delta
\]
To construct the maps $\eta_n \colon S^n \to KU^A_n$, note that $t \mapsto t\,c$ is an odd, self-adjoint, regular unbounded multiplier on $C_0(\R, \C\ell_1)$. Therefore the functional calculus for this multiplier is a graded $*$-homomorphism $\widehat{S} \to C_0(\R, \C\ell_1)$. This in turn can be seen as a basepoint preserving map $S^1 \to \hom_{\rm gr}(\widehat{S}, \C\ell_1)$. Now consider
\[
	C_0(\R, \C\ell_1) \to C_0(\R, \C\ell_1 \otimes A \otimes \K) \quad ; \quad f \mapsto f \otimes (1 \otimes e)\ ,
\]
where $e$ is a rank $1$-projection in $\K$. After concatenation, we obtain a graded $*$-homomor\-phism $\hat{\eta}_1 \colon \widehat{S} \to C_0(\R, \C\ell_1 \otimes A \otimes \K)$ and from this a continuous map $\eta_1 \colon S^1 \to \hom_{\rm gr}(\widehat{S}, \C\ell_1 \otimes A \otimes \K)$. Now let 
\[
	\hat{\eta}_n \colon \widehat{S} \to C_0(\R^n, (\C\ell_1 \otimes A \otimes \K)^{\widehat{\otimes} n}) \quad \text{with } \hat{\eta}_n = (\hat{\eta}_1 \widehat{\otimes} \dots \widehat{\otimes}\,\hat{\eta}_1) \circ \Delta_n\ ,
\]
where $\Delta_n \colon \widehat{S} \to \widehat{S}^{\,\widehat{\otimes} n}$ is defined recursively by $\Delta_1 = \id{\widehat{S}}$, $\Delta_2=\Delta$ and $\Delta_n = (\Delta\,\widehat{\otimes}\,\id{})\circ \Delta_{n-1}$, for $n\geq 3$. This yields a well-defined map $\eta_n \colon S^n \to KU^A_n$.

\begin{theorem} \label{thm:KUA_spectrum}
Let $A$ be a separable, unital, strongly self-absorbing $C^*$-algebra. The spaces $KU^A_{\bullet}$ together with the maps $\mu_{\bullet,\bullet}$ and $\eta_{\bullet}$ form a commutative symmetric ring spectrum with coefficients
\[
	\pi_n(KU^A_{\bullet}) = K_n(A)\ .
\]
Moreover, all structure maps $KU^A_n \to \Omega KU^A_{n+1}$ are weak homotopy equivalences for $n \geq 1$.
\end{theorem}

\begin{proof}\label{pf:KUA_spectrum}
It is a consequence of the cocommutativity of $\Delta$ that $\Delta_n$ is $\Sigma_n$-invariant, i.e.\ for a permutation $\tau \in \Sigma_n$ and its induced action $\tau_* \colon \widehat{S}^{\,\widehat{\otimes} n} \to \widehat{S}^{\,\widehat{\otimes} n}$ we have $\tau_* \circ \Delta_n = \Delta_n$. If we let $\tau$ act on 
\[
	C_0(\R^n, (\C\ell_1 \otimes A \otimes \K)^{\,\widehat{\otimes} n}) \cong (C_0(\R, \C\ell_1 \otimes A \otimes \K))^{\widehat{\otimes} n}
\]	
by permuting the tensor factors, then we have that $\hat{\eta}_n$ is invariant under the action of $\Sigma_n$ since
\[
	\tau_* \circ \hat{\eta}_n = \tau_* \circ (\hat{\eta}_1 \widehat{\otimes} \dots \widehat{\otimes}\,\hat{\eta}_1) \circ \Delta_n = (\hat{\eta}_1 \widehat{\otimes} \dots \widehat{\otimes}\,\hat{\eta}_1) \circ \tau_* \circ \Delta_n = \hat{\eta}_n\ ,
\]
which proves that $\eta_n$ is $\Sigma_n$-\emph{equi}variant. The $\Sigma_m \times \Sigma_n$-equivariance of $\mu_{m,n}$ is clear from the definition of the $\Sigma_{m+n}$-action on $KU^A_{m+n}$ and the symmetry of the graded tensor product. Associativity of $\mu_{\bullet,\bullet}$ (see Definition \ref{def:symmetric_spectrum} (a)) is a direct consequence of the coassociativity of $\widehat{S}$ and the associativity of the graded tensor product. The map $\mu_{m,n}(\eta_m \wedge \eta_n)$ corresponds to the $*$-homomorphism
\[
	(\hat{\eta}_m \widehat{\otimes} \,\hat{\eta}_n) \circ \Delta = \underbrace{\hat{\eta}_1 \widehat{\otimes} \dots \widehat{\otimes}\,\hat{\eta}_1}_{n+m \text{ times}} \circ (\Delta_m \widehat{\otimes}\, \Delta_n) \circ \Delta\ .
\]
By coassociativity we have $(\Delta_m \widehat{\otimes}\, \Delta_n) \circ \Delta = \Delta_{m+n}$. Therefore $(\hat{\eta}_m \widehat{\otimes} \,\hat{\eta}_n) \circ \Delta = \hat{\eta}_{m+n}$, which translates into the compatibility condition of Definition \ref{def:symmetric_spectrum}. The commutativity of $\mu_{\bullet, \bullet}$ follows from the definition of the permutation action and the cocommutativity of the coalgebra structure on $\widehat{S}$. Thus, the spaces $KU^A_{\bullet}$ indeed form a commutative symmetric ring spectrum.

To see that the structure maps induce weak equivalences, observe that for $k \geq 1$
\[
	\pi_k(KU^A_n) = \pi_0(\Omega^k KU^A_n) = \pi_0(\hom_{\rm gr}(\widehat{S}, C_0(\R^k, (\C\ell_1 \otimes \K \otimes A)^{\widehat{\otimes} n})))
\]
The algebra $\C\ell_1 \otimes \K$ with trivially graded $\K$ is isomorphic as a graded $C^*$-algebra to $\C\ell_1 \widehat{\otimes} \widehat{\K}$, where $\widehat{\K} = \widehat{\K}(H_+ \oplus H_-)$ is equipped with its standard even grading. Therefore 
\[
	C_0(\R^k, (\C\ell_1 \otimes \K \otimes A)^{\widehat{\otimes} n}) \cong C_0(\R^k, (\C\ell_1 \otimes A)^{\widehat{\otimes} n}) \widehat{\otimes}\,\widehat{\K}^{\widehat{\otimes} n} \cong C_0(\R^k, (\C\ell_1 \otimes A)^{\widehat{\otimes} n}) \widehat{\otimes}\,\widehat{\K}
\] 
and by \cite[Theorem 4.7]{paper:Trout}
\[
	\pi_0(\hom_{\rm gr}(\widehat{S}, C_0(\R^k, (\C\ell_1 \otimes A)^{\widehat{\otimes} n}) \widehat{\otimes}\,\widehat{\K}) \cong KK(\C, C_0(\R^k, (\C\ell_1 \otimes A)^{\widehat{\otimes} n})) \cong K_{k-n}(A)
\]
where we used Bott periodicity and the isomorphism $A^{\otimes n} \cong A$ in the last map.  The structure map $KU^A_n \to \Omega KU^A_{n+1}$ is now given by exterior multiplication with the class in $KK(\C, C_0(\R, \C\ell_1 \otimes A))$ represented by $\widehat{\eta}_1$. But by definition this is a combination of the Bott element together with the map $A \otimes \K \to (A \otimes \K)^{\otimes 2}$ that sends $a$ to $a \otimes (1 \otimes e)$, which is homotopic to an isomorphism, since $A$ is \emph{strongly} self-absorbing. Thus, both operations induce isomorphisms on $K$-theory and therefore also on all homotopy groups. Finally, we have for $n \in \Z$
\[
	\pi_n(KU^A_{\bullet}) = \colim_{k \to \infty} \pi_{n+k}(KU^A_{k}) \cong K_{n}(A)\ . \qedhere
\]
\end{proof}

\begin{remark}
As can be seen from the above proof, the fact that $A$ is \emph{strongly} self-absorbing (and not just self-absorbing, which would be $A \cong A \otimes A$) is important for the spectrum to be positive fibrant.
\end{remark}	
	
\begin{remark}
Note that the spectrum $KU^A_{\bullet}$ is not connective. However, the $\I$-monoid $\units{R}$ of a symmetric ring spectrum $R$ only ``sees'' the connective cover. In fact, $\Omega^{\infty}(R)_{h\I}$ is its underlying infinite loop space. This was already remarked in \cite[Section 1.6]{paper:SagaveSchlichtkrull} and was the motivation for the introduction of $\mathcal{J}$-spaces to also capture periodic phenomena. 
\end{remark}

\subsection{The Eckmann-Hilton $\I$-group $G_A$}
Let $A$ be a separable, unital, strongly self-absorbing algebra and let $G_A(\mathbf{n}) = \Aut{(A \otimes \K)^{\otimes n}}$ (with $(A \otimes \K)^{\otimes 0} := \C$, such that $G_A(\mathbf{0})$ is the trivial group). Note that $\sigma \in \Sigma_n$ acts on $G_A(\mathbf{n})$ by mapping $g \in G_A(\mathbf{n})$ to $\sigma \circ g \circ \sigma^{-1}$, where $\sigma$ permutes the tensor factors of $(A \otimes \K)^{\otimes n}$. 
Let $g \in G_A(\mathbf{m})$ and $\alpha \colon \mathbf{m} \to \mathbf{n}$. We can enhance $G_A$ to a functor $G_A \colon \I \to \mathcal{G}rp_{\ast}$ via
\[
	\alpha_*(g) = \bar{\alpha} \circ (\id{(A \otimes \K)^{\otimes (n-m)}} \otimes g) \circ \bar{\alpha}^{-1}
\]
with $\bar{\alpha}$ as explained after Definition \ref{def:ISpace_IMonoid}. Moreover, $G_A$ is an $\I$-monoid via
\[
	\mu_{m,n} \colon G_A(\mathbf{m}) \times G_A(\mathbf{n}) \to G_A(\mathbf{m} \sqcup \mathbf{n}) \quad ; \quad (g,h) \mapsto g \otimes h\ .
\]
Note that $\mu_{m,0}$ and $\mu_{0,n}$ are induced by the canonical isomorphisms $\C \otimes (A \otimes \K)^{\otimes n} \cong (A \otimes \K)^{\otimes n}$ and $(A \otimes \K)^{\otimes m} \otimes \C \cong (A \otimes \K)^{\otimes m}$ respectively.

\begin{theorem}\label{thm:A_is_EHI}
$G_A$ as defined above is a stable EH-$\I$-group with compatible inverses in the sense of Definition \ref{def:EHI_group}.
\end{theorem}

\begin{proof} \label{pf:A_is_EHI}
The commutativity of the Eckmann-Hilton diagram in Definition \ref{def:EHI_group} is a consequence of $(g \otimes h) \cdot (g' \otimes h') = (g \cdot g') \otimes (h \cdot h')$ for $g,g' \in G_A(\mathbf{m})$ and $h,h' \in G_A(\mathbf{n})$. Well-pointedness of $G_A(\mathbf{n})$ was proven in Corollary 
\cite[Prop.2.6]{paper:DadarlatP1}. To see that non-initial morphisms are mapped to homotopy equivalences, first observe that $G_A$ maps permutations to homeomorphisms, therefore it suffices to check that 
\[
	\Aut{A \otimes \K} \to \Aut{(A \otimes \K)^{\otimes n}} \quad ; \quad g \mapsto g \otimes \id{(A \otimes \K)^{\otimes (n-1)}}
\]
is a homotopy equivalence. But the target space is homeomorphic to $\Aut{A \otimes \K}$ via conjugation with a suitable isomorphism. Altogether it is enough to check that we can find an isomorphism $\psi \colon A \otimes \K \to (A \otimes \K)^{\otimes 2}$ such that 
\[
	\Aut{A \otimes \K} \to \Aut{A \otimes \K} \quad ; \quad g \mapsto \psi^{-1} \circ (g \otimes \id{A \otimes \K}) \circ \psi
\]
is homotopic to $\id{\Aut{A \otimes \K}}$. Let $\hat{\psi} \colon I \times (A \otimes \K) \to (A \otimes \K)^{\otimes 2}$ be the homotopy of \cite[Thm.2.5]{paper:DadarlatP1} and let $\psi = \hat{\psi}(\frac{1}{2})$. Then
\[
	H \colon I \times \Aut{A \otimes \K} \to \Aut{A \otimes \K} \  ; \  
	g \mapsto 
	\begin{cases}
		\hat{\psi}(\frac{t}{2})^{-1} \circ (g \otimes \id{A \otimes \K}) \circ \hat{\psi}(\frac{t}{2}) & \text{for } 0 < t \leq 1 \\
		g & \text{for } t = 0
	\end{cases}
\] 
satisfies the conditions. 

Let $g \in G_A(\mathbf{m})$. To see that there is a path connecting $(\iota_m \sqcup \id{\mathbf{m}})_*(g) = \id{(A \otimes \K)^{\otimes m}} \otimes g$ to $g \otimes \id{(A \otimes \K)^{\otimes m}}$, observe that $\tau_A \colon A \otimes A \to A \otimes A$ with $\tau_A(a \otimes b) = b \otimes a$ is homotopic to the identity by the contractibility of $\Aut{A \otimes A} \cong \Aut{A}$ \cite[Thm.2.3]{paper:DadarlatP1}. Similarly, there is a homotopy between $\tau_{\K} \colon \K \otimes \K \to \K \otimes \K$ with $\tau_{\K}(S \otimes T) = T \otimes S$ and the identity since $\Aut{\K \otimes \K} \cong PU(H)$ is path-connected. 

Thus, we obtain a homotopy between the identity on $(A \otimes \K)^{\otimes m} \otimes (A \otimes \K)^{\otimes m}$ and the corresponding switch automorphism $\tau_{m} \colon (A \otimes \K)^{\otimes m} \otimes (A \otimes \K)^{\otimes m} \to (A \otimes \K)^{\otimes m} \otimes (A \otimes \K)^{\otimes m}$. But, $\tau_m \circ (\id{(A \otimes \K)^{\otimes m}} \otimes g) \circ \tau_m^{-1} = g \otimes \id{(A \otimes \K)^{\otimes m}}$.
\end{proof}

Recall that $KU^A_n = \hom_{\rm gr}(\widehat{S}, (\C\ell_1 \otimes A \otimes \K)^{\widehat{\otimes} n})$. There is a unique isomorphism, which preserves the order of factors of the same type:
\[
	\theta_n \colon (\C\ell_1 \otimes A \otimes \K)^{\widehat{\otimes} n} \to (\C\ell_1)^{\widehat{\otimes} n} \otimes (A \otimes \K)^{\otimes n}\ ,
\]
where -- as above -- the grading on both sides arises from the grading of $\C\ell_1$ and $A \otimes \K$ is trivially graded. In particular, if $\sigma \in \Sigma_n$ is a permutation and $\sigma_* \colon (\C\ell_1 \otimes A \otimes \K)^{\widehat{\otimes} n}\to (\C\ell_1 \otimes A \otimes \K)^{\widehat{\otimes} n}$ is the operation permuting the factors of the graded tensor product, then 
\begin{equation} \label{eqn:theta_n}
	\theta_n \circ \sigma_* = (\sigma^{\C\ell_1}_* \otimes \sigma^{A \otimes \K}_*) \circ \theta_n
\end{equation}
where $\sigma^{\C\ell_1}_* \colon (\C\ell_1)^{\widehat{\otimes} n} \to (\C\ell_1)^{\widehat{\otimes} n}$ and $\sigma^{A \otimes \K}_* \colon (A \otimes \K)^{\otimes n} \to (A \otimes \K)^{\otimes n}$ are the corresponding permutations. Thus, we can define an action of $G_A(\mathbf{n})$ on $KU^A_n$ via 
\[
	\kappa_n \colon G_A(\mathbf{n}) \times KU^A_n \to KU^A_n \quad ; \quad (g, \varphi) \mapsto \theta_n^{-1} \circ (\id{\C\ell_1^{\widehat{\otimes} n}} \otimes g) \circ \theta_n \circ \varphi 
\]

\begin{theorem}\label{thm:HigherTwists}
Let $A$ be a separable, unital, strongly self-absorbing $C^*$-algebra. Then the EH-$\I$-group $G_A$ associated to $A$ acts on the commutative symmetric ring spectrum $KU^A_{\bullet}$ via $\kappa_{\bullet}$ as defined above. We obtain a map of very special $\Gamma$-spaces $\Gamma(G_A) \to \Gamma(\units{KU^{A}})$, which induces an isomorphism on all homotopy groups $\pi_n$ of the associated spectra with $n > 0$ and the inclusion $K_0(A)^{\times}_+ \to K_0(A)^{\times}$ on $\pi_0$. If $A$ is purely infinite \cite[Def.4.1.2]{book:RoerdamStoermer}, this is an equivalence in the stable homotopy category. 
\end{theorem}

\begin{proof}\label{pf:HigherTwists}
It is clear that $\kappa_n$ preserves the basepoint of $KU^A_n$. The $\Sigma_n$-equivariance is a direct consequence of (\ref{eqn:theta_n}). Indeed, we have for a permutation $\sigma \in \Sigma_n$, $g \in G_A(\mathbf{n})$ and $\varphi \in KU^A_n$:
\begin{align*}
 & \kappa_n(\sigma^{A \otimes \K} \circ g \circ (\sigma^{A \otimes \K})^{-1}, \sigma \circ \varphi)	\\
= \ & \theta_n^{-1} \circ (\id{} \otimes (\sigma^{A \otimes \K} \circ g \circ (\sigma^{A \otimes \K})^{-1})) \circ \theta_n \circ \sigma \circ \varphi \\
	= \ & \theta_n^{-1} \circ (\sigma^{\C\ell_1} \otimes \sigma^{A \otimes \K}) \circ (\id{} \otimes g) \circ (\sigma^{\C\ell_1} \otimes \sigma^{A \otimes \K})^{-1} \circ \theta_n \circ \sigma \circ \varphi \\
	= \ & \sigma \circ \theta_n^{-1} \circ (\id{} \otimes g) \circ \theta_n \circ \varphi = \sigma \circ \kappa_n(g, \varphi)\ ,
\end{align*}
where we omitted the stars from the notation. 

Let $\tau \colon \C\ell_1^{\widehat{\otimes} m} \otimes (A \otimes \K)^{\otimes m}\,\widehat{\otimes}\,\C\ell_1^{\widehat{\otimes} n} \otimes (A \otimes \K)^{\otimes n} \to \C\ell_1^{\widehat{\otimes} m+n}\,\widehat{\otimes}\,(A \otimes \K)^{\otimes m+n}$ be the permutation of the two middle factors, then $\tau \circ (\theta_m\,\widehat{\otimes}\,\theta_n) = \theta_{m+n}$. This implies that
\[
	(\kappa_m(g, \varphi) \,\widehat{\otimes}\, \kappa_n(h,\psi)) \circ \Delta = \kappa_{m+n}(g \otimes h, \,(\varphi \,\widehat{\otimes}\,\psi) \circ \Delta)
\]
for $g \in G_A(\mathbf{m}), h \in G_A(\mathbf{n})$, $\varphi \in KU^A_m, \psi \in KU^A_n$, which is the compatibility condition in Definition \ref{def:EHI_action} (\ref{it:mult}). The same argument shows that for $l + m = n$, $g \in G(\mathbf{m})$ and $\varphi \in KU^A_m$
\[
	(\widehat{\eta}_l \,\widehat{\otimes}\, \kappa_m(g, \varphi)) \circ \Delta = \kappa_n(\id{} \otimes g, (\widehat{\eta}_l \,\widehat{\otimes}\,\varphi) \circ \Delta)\ ,
\]
which is the crucial observation to see that diagram (\ref{it:stab}) in Definition~\ref{def:EHI_action} commutes.

Thus, we have proven that $G=G_A$ acts on the spectrum $KU^A$. By Theorem \ref{thm:EHI_action} together with Theorem \ref{thm:EHI_group} we obtain a map of $\Gamma$-spaces
\(
	\Gamma(B_{\nu}G_A) \to B_{\mu}\Gamma(G_A) \to B_{\mu} \Gamma(\units{KU^A})
\)
where the first map is a strict equivalence. We see from Lemma \ref{lem:h_inverse} that $G_{h\I}$ is in fact a grouplike topological monoid, i.e.\ $G_{h\I} \to \Omega B_{\mu} G_{h\I}$ is a homotopy equivalence. Thus, to finish the proof, we only need to check that $G_{h\I} \to GL_1(KU^A) = \units{KU^A}_{h\I}$ has the desired properties. Consider the diagram
\[
	\xymatrix{
		G_{h\I} \ar[r] & GL_1(KU^A) \\
		G(\mathbf{1}) \ar[r] \ar[u]^-{\simeq} & \units{KU^A}(\mathbf{1}) \ar[u]_-{\simeq}
	}
\]
where the vertical maps are given by the inclusions into the zero skeleton of the homotopy colimit. The latter are equivalences by \cite[Lemma 2.1]{paper:Schlichtkrull}. By this lemma, we also see that $\pi_0(GL_1(KU^A)) = GL_1(\pi_0(KU^A)) = GL_1(K_0(A)) = K_0(A)^{\times}$. It remains to be seen that 
\[
	\Theta \colon \Aut{A \otimes \K} \to \Omega KU^A_1 = \hom_{\rm gr}(\widehat{S}, C_0(\R, \C\ell_1) \otimes A \otimes \K) \quad ; \quad g \mapsto \kappa_1(g, \widehat{\eta}_1) 
\]
induces an isomorphism on all homotopy groups $\pi_n$ with $n > 0$ and the inclusion $K_0(A)^{\times}_+ \to K_0(A)$ on $\pi_0$. The basepoint of the target space is now given by $\widehat{\eta}_1$ instead of the zero homomorphism. $\Theta$ fits into a commutative diagram
\[
	\xymatrix{
	\Aut{A \otimes \K} \ar[rr]^-{\Theta} \ar[dr]_-{\Phi} & & \hom_{\rm gr}(\widehat{S}, C_0(\R, \C\ell_1) \otimes A \otimes \K) \\
	& \Proj{A \otimes \K}^{\times} \ar[ur]_-{\Psi}
	}
\]
with $\Phi(g) = g(1 \otimes e)$ and $\Psi(p) = \widetilde{\eta}_1 \otimes p$, where $\widetilde{\eta}_1 \in \hom_{\rm gr}(\widehat{S}, C_0(\R, \C\ell_1))$ arises from the functional calculus of the operator described after Definition \ref{def:KUA_spectrum}. It was shown in \cite[Thm.2.16, Thm.2.5]{paper:DadarlatP1} that $\Phi$ is a homotopy equivalence. Let $\epsilon \in \hom_{\rm gr}(\widehat{S}, \C)$ be the counit of $\widehat{S}$. The map $\Psi$ factors as
\[
	\Psi \colon \Proj{A \otimes \K}^{\times} \to \hom_{\rm gr}(\widehat{S}, A\,\widehat{\otimes}\,\widehat{\K}) \to \hom_{\rm gr}(\widehat{S}, C_0(\R,\C\ell_1) \otimes A \otimes \K)
\]
where the first map sends a projection $p$ to $\epsilon \cdot \left(\begin{smallmatrix} p & 0 \\ 0 & 0 \end{smallmatrix}\right)$ and the second sends $\varphi$ to $\varphi \,\widehat{\otimes}\, \widetilde{\eta}_1 \circ \Delta$ and applies the graded isomorphism to shift the grading to $C_0(\R, \C\ell_1)$ only. Since the second map induces multiplication with the Bott element, it is an isomorphism on $\pi_0$. It was proven in \cite[Cor.2.17]{paper:DadarlatP1} that $\pi_0(\Proj{A \otimes \K}^{\times}) \cong K_0(A)^{\times}_+$. The discussion after the proof of \cite[Thm.4.7]{paper:Trout} shows that the first map induces the inclusion $K_0(A)^{\times}_+ \to K_0(A)$ on $\pi_0$.

Let $B$ be a graded, $\sigma$-unital $C^*$-algebra and define $K'_n(B)$ to be the kernel of the map $K'(C(S^n) \otimes B) \to K'(B)$ induced by evaluation at the basepoint. Here, we used the notation $K'(B) = \pi_0(\hom_{\rm gr}(\widehat{S}, B \,\widehat{\otimes}\,\widehat{K}))$ introduced in \cite{paper:Trout}. The five lemma shows that $K'_n(B)$ is in fact isomorphic to $K_n(B)$, if we identify the latter with the kernel $K_0(C(S^n) \otimes B) \to K_0(B)$. For $n > 0$ we have the commutative diagram
\[
	\xymatrix{
	\pi_n(\Proj{A \otimes \K}^{\times}, 1 \otimes e) \ar[r]^-{\Psi_*} \ar[d]_-{\cong} & \pi_n(\Omega KU^A_1, \widehat{\eta}_1) \ar[d]^-{\cong} \\
	K_n(A) \ar[r]_-{\cong} & K_n'(C_0(\R, \C\ell_1) \otimes A)
	}
\]
where the lower horizontal map sends $[p]-[q] \in K_n(A) = \ker(K_0(C(S^n) \otimes A) \to K_0(A))$ to $\left[\widetilde{\eta}_1\,\widehat{\otimes}\,\left(\begin{smallmatrix} p & 0 \\ 0 & q \end{smallmatrix}\right)\right] \in K_n'(C_0(\R, \C\ell_1) \otimes A)$. The same argument as for $\pi_0$ above shows that this is an isomorphism. Every element $\gamma \colon (S^n,x_0) \to (\Proj{A \otimes \K}, 1\otimes e)$ induces a projection $p_{\gamma} \in C(S^n) \otimes A \otimes \K$. The vertical map on the left sends $[\gamma]$ to $[p_{\gamma}] - [1_{C(S^n)} \otimes 1 \otimes e] \in K_0(C(S^n,x_0)\otimes A)\cong K_n(A)$. This map is an isomorphism by Bott periodicity. Finally, we can consider $\gamma' \colon S^n \to \Omega KU^A_1$ as an element $\varphi_{\gamma'} \in \hom_{\rm gr}(\widehat{S}, C(S^n) \otimes C_0(\R,\C\ell_1) \otimes A \otimes \K)$ ($\K$ trivially graded!). The vertical map on the right hand side sends $[\gamma']$ to $\left[\left( \begin{smallmatrix} \varphi_{\gamma'} & 0 \\ 0 & 1_{C(S^n)} \otimes \,\widehat{\eta}_1 \end{smallmatrix} \right)\right] \in K_n'(C_0(\R,\C\ell_1) \otimes A)$. It corresponds to the `subtraction of $1$`, i.e.\ it corrects the basepoint by shifting back to the component of the zero homomorphism. Its inverse is given by $\psi \mapsto \psi \oplus \left(\begin{smallmatrix} 1_{C(S^n)} \otimes\,\widehat{\eta}_1 & 0 \\ 0 & 0 \end{smallmatrix}\right)$, where $\oplus$ is the addition operation described in \cite{paper:Trout}. 

If $A$ is purely infinite, we have $K_0(A)^{\times}_+ = K_0(A)^{\times}$, which implies the last statement. This finishes the proof.
\end{proof}
	
\subsection{Applications}
The infinite Cuntz algebra $\Cuntz{\infty}$ is the universal unital $C^*$-algebra generated by countably infinitely many generators $s_i$ that satisfy the relations $s_i^*s_j = \delta_{ij}\,1$. It is purely infinite, strongly self-absorbing, satisfies the universal coefficient theorem in KK-theory and for any locally compact Hausdorff space $X$ the unit homomorphism $\C \to \Cuntz{\infty}$ induces a natural isomorphism of multiplicative cohomology theories $K^0(X) = K_0(C(X)) \to K_0(C(X) \otimes \Cuntz{\infty})$.

\begin{theorem} \label{thm:CuntzHigherK}
The very special $\Gamma$-spaces $\Gamma(G_{\Cuntz{\infty}})$ and $\Gamma(\units{KU})$ are strictly equivalent, which implies that the spectrum associated to $\Gamma(G_{\Cuntz{\infty}})$ is equivalent to $gl_1(KU)$ in the stable homotopy category. In particular, $B\Aut{\Cuntz{\infty} \otimes \K}$ is weakly homotopy equivalent to $BBU_{\otimes} \times B(\Z / 2\Z)$. 
\end{theorem}

\begin{proof}	
The unit homomorphism $\C \to \Cuntz{\infty}$ yields $\Gamma(\units{KU^{\C}}) \to \Gamma(\units{KU^{\Cuntz{\infty}}})$. To see that this is a strict equivalence, it suffices to check that $GL_1(KU) \to GL_1(KU^{\Cuntz{\infty}})$ is a weak equivalence. Let $X$ be a finite CW-complex. The isomorphism $[X,\Omega^1(KU_1^{\Cuntz{\infty}})] \cong K_0(C(X) \otimes \Cuntz{\infty})$ constructed above restricts to $[X,GL_1(KU^{\Cuntz{\infty}})] \cong GL_1(K_0(C(X) \otimes \Cuntz{\infty}))$ and similarly with $\C$ instead of $\Cuntz{\infty}$. The composition
\[
	GL_1(K^0(X)) \cong [X, GL_1(KU)] \to [X, GL_1(KU^{\Cuntz{\infty}})] \cong GL_1(K_0(C(X) \otimes \Cuntz{\infty})) 
\]
is the restriction of the ring isomorphism $K^0(X) \to K_0(C(X) \otimes \Cuntz{\infty})$ to the invertible elements. 

By Theorem \ref{thm:HigherTwists} we obtain a strict equivalence $\Gamma(G_{\Cuntz{\infty}}) \to \Gamma(\units{KU^{\Cuntz{\infty}}})$. Therefore the zig-zag $\Gamma(G_{\Cuntz{\infty}}) \to \Gamma(\units{KU^{\Cuntz{\infty}}}) \leftarrow \Gamma(\units{KU})$ proves the first claim. From this, we get a weak equivalence between $BGL_1(KU) \simeq BBU_{\otimes} \times B(\Z/2\Z)$ and $B_{\nu}G_A(\mathbf{1}) = B\Aut{\Cuntz{\infty} \otimes \K}$ using Lemma \ref{lem:Hocolim_is_BG} and the stability of $G_A$.
\end{proof}

The UHF-algebra $M_{p^{\infty}}$ is constructed as an infinite tensor product of matrix algebras $M_{p}(\C)$. It is separable, unital, strongly self-absorbing with $K_0(M_{p^{\infty}}) \cong \Z[\tfrac{1}{p}]$, $K_1(M_{p^{\infty}}) = 0$. Likewise, if we fix a prime $p$ and choose a sequence $(d_j)_{j \in \N}$ such that each prime number except $p$ appears in the sequence infinitely many times, we can recursively define $A_{n+1} = A_n \otimes M_{d_n}$, $A_0 = \C$. The direct limit $M_{(p)} = \lim A_n$ is a separable, unital, strongly self-absorbing $C^*$-algebra with $K_0(M_{(p)}) \cong \Z_{(p)}$ -- the localization of the integers at $p$ -- and $K_1(M_{(p)}) = 0$.

\begin{theorem} \label{thm:localization}
The very special $\Gamma$-spaces $\Gamma(G_{M_{(p)} \otimes \Cuntz{\infty}})$ and $\Gamma(\units{KU_{(p)}})$ are strictly equivalent, which implies that the spectrum associated to $\Gamma(G_{M_{(p)} \otimes \Cuntz{\infty}})$ is equivalent to $gl_1(KU_{(p)})$ in the stable homotopy category. In particular, $B\Aut{M_{(p)} \otimes \Cuntz{\infty} \otimes \K}$ is weakly homotopy equivalent to $GL_1(KU_{(p)})$. The analogous statement for the localization \emph{away} from $p$ is also true if $M_{(p)}$ is replaced by $M_{p^{\infty}}$.
\end{theorem}

\begin{proof}
Consider the commutative symmetric ring spectrum $KU^{M_{(p)} \otimes \Cuntz{\infty}}_{\bullet}$ with homotopy groups $\pi_{2k}(KU^{M_{(p)} \otimes \Cuntz{\infty}}) = \Z_{(p)}$, $\pi_{2k+1}(KU^{M_{(p)} \otimes \Cuntz{\infty}}) = 0$. Note that the unit homomorphism $\C \to M_{(p)} \otimes \Cuntz{\infty}$ induces a map of spectra $KU^{\C} \to KU^{M_{(p)}\otimes \Cuntz{\infty}}$, which is the localization map $\Z \to \Z_{(p)}$ on the non-zero coefficient groups. Let $S_{(p)}$ be the Moore spectrum, i.e.\ the commutative symmetric ring spectrum with $\widetilde{H}^1(S_{(p)}; \Z) \cong \Z_{(p)}$ and $\widetilde{H}^k(S_{(p)};\Z) = 0$ for $k \neq 1$. We have $KU_{(p)} = KU \wedge S_{(p)}$ and $\pi_k(KU) \to \pi_k(KU_{(p)}) = \pi_k(KU) \otimes \Z_{(p)}$ is the localization map. It follows that $KU_{(p)} \to KU^{M_{(p)} \otimes \Cuntz{\infty}} \wedge S_{(p)} \leftarrow KU^{M_{(p)}\otimes \Cuntz{\infty}}$ is a zig-zag of $\pi_*$-equivalences. Just as in Theorem \ref{thm:CuntzHigherK} we show that it induces a zig-zag of strict equivalences $\Gamma(\units{KU_{(p)}}) \to \Gamma(\units{KU^{M_{(p)} \otimes \Cuntz{\infty}} \wedge S_{(p)}}) \leftarrow \Gamma(\units{KU^{M_{(p)} \otimes \Cuntz{\infty}}})$, which shows that $gl_1(KU_{(p)})$ is stably equivalent to $gl_1(KU^{M_{(p)} \otimes \Cuntz{\infty}})$. But by Theorem \ref{thm:HigherTwists} we have a strict equivalence $\Gamma(G_{M_{(p)}\otimes \Cuntz{\infty}}) \to \Gamma(\units{KU^{M_{(p)}\otimes \Cuntz{\infty}}})$. The proof for the localization away from $p$ is completely analogous therefore we omit it.
\end{proof}

In \cite{paper:DadarlatP1} the authors used a permutative category $\mathcal{B}_{\otimes}$ to show that $B\Aut{A \otimes \K}$ carries an infinite loop space structure: The objects of $\mathcal{B}_{\otimes}$ are the natural numbers $\N_0 = \{0, 1, 2, \dots\}$ and $\hom(m,n) = \{\alpha \in \hom((A \otimes \K)^{\otimes m}, (A \otimes \K)^{\otimes n})\ |\ \alpha((1 \otimes e)^{\otimes m}) \in GL_1(K_0((A \otimes \K)^{\otimes n}))\}$, where $e\in \K$ is a rank 1-projection and $(1 \otimes e)^{\otimes 0} = 1 \in \C$. Since $\hom(0,1)$ is non-empty, there is a stabilization $\theta$ of the object $1 \in \obj(\mathcal{B}_{\otimes})$ by Lemma \ref{lem:bij_stabi}. In fact, we may choose $\theta(\iota_{\mathbf{1}}) \in \hom(\C, A \otimes \K)$ to be $\theta(\lambda) = \lambda(1 \otimes e)$.

\begin{theorem}
Let $A$ be a separable, strongly self-absorbing $C^*$-algebra. Then there is a strict equivalence of $\Gamma$-spaces $\Gamma(B_{\nu}G_A) \to \Gamma(\mathcal{B}_{\otimes})$. In particular, the induced infinite loop space structures on $B\Aut{A \otimes \K}$ agree.
\end{theorem}

\begin{proof}
Note that ${\rm Aut}_{\mathcal{B}_{\otimes}}(1) = \Aut{A \otimes \K}$ and that the maps $B_{\nu}G_A(\mathbf{m}) = B\Aut{(A \otimes \K)^{\otimes m}} \to \lvert N_{\bullet}\mathcal{B}_{\otimes} \rvert$ are homotopy equivalences for $m > 0$. Thus, the statement follows from Theorem \ref{thm:Schlichtkrull_vs_Segal}.
\end{proof}

The above identification may be used to prove theorems about bundles (and continuous fields) of strongly self-absorbing $C^*$-algebras using what is known about the unit spectrum of $K$-theory. As an example let $X$ be a compact metrizable space and consider the cohomology group $[X, B\uAut{\Cuntz{\infty}\otimes \K}]$, where we use the notation of \cite{paper:DadarlatP1}. Note that the third Postnikov section of $B\uAut{\Cuntz{\infty} \otimes \K}$ is a $K(\Z,3)$, let $B\uAut{\Cuntz{\infty} \otimes \K} \to K(\Z,3)$ be the corresponding map and denote by $F$ its homotopy fiber. The composition $B\Aut{\K} \to B\uAut{\Cuntz{\infty}\otimes \K} \to K(\Z,3)$, where the first map is induced by the unit homomorphism $\C \to \Cuntz{\infty}$, is a homotopy equivalence. Therefore we obtain a homotopy splitting 
\[
	B\Aut{\K} \times F \overset{\simeq}{\longrightarrow} B\uAut{\Cuntz{\infty}\otimes \K}
\]
and a corresponding fibration $B\Aut{\K} \to B\uAut{\Cuntz{\infty}\otimes \K} \to F$. The weak equivalence between $B\uAut{\Cuntz{\infty} \otimes \K}$ and $BBU_{\otimes} \simeq BBU(1) \times BBSU_{\otimes}$ identifies $F$ with the corresponding homotopy fiber of the third Postnikov section of $BBU_{\otimes}$, which is $BBSU_{\otimes}$. Thus, we obtain a short exact sequence
\[
	0 \to [X, B\Aut{\K}] \to [X, B\uAut{\Cuntz{\infty} \otimes \K}] \to [X, F] \cong bsu^1_{\otimes}(X) \to 0
\]
The following was proven by G\'{o}mez in \cite[Theorem 5]{paper:Gomez}.
\begin{theorem}\label{thm:TheNoTwist}
Let $G$ be a compact Lie group. Then $bsu_{\otimes}^1(BG) = 0$.
\end{theorem}
In light of our previous results, we obtain the Corollary:
\begin{corollary}
Let $X$ be a compact metrizable space, let $G$ be a compact Lie group and let $P \to X$ be a principal $G$-bundle. Let $\alpha \colon G \to \uAut{\Cuntz{\infty} \otimes \K}$ be a continuous homomorphism. Then the associated $\uAut{\Cuntz{\infty} \otimes \K}$-bundle $Q \to X$ given by
\(
	Q = P \times_{\alpha} \uAut{\Cuntz{\infty} \otimes \K}
\)
is isomorphic to $\widetilde{Q} \times_{\rm Ad} \uAut{\Cuntz{\infty} \otimes \K}$ for a principal $PU(H)$-bundle $\widetilde{Q}$.
\end{corollary}

\begin{proof}
The class $[Q] \in [X, B\uAut{\Cuntz{\infty} \otimes \K}]$ is the pullback of $[B\alpha] \in [BG, B\uAut{\Cuntz{\infty}\otimes \K}]$ via the classifying map $f_P \colon X \to BG$ of $P$. The diagram
\[
	\xymatrix{
	& & [BG, B\uAut{\Cuntz{\infty}\otimes \K}] \ar[d] \ar[r] & [BG,F] \cong bsu^1_{\otimes}(BG)\ar[d] \\
	0 \ar[r] & [X, B\Aut{\K}] \ar[r] & [X, B\uAut{\Cuntz{\infty}\otimes \K}] \ar[r] & [X,F] \ar[r] & 0
	}
\]
shows that $[Q]$ is mapped to $0$ in $[X,F] \cong bsu^1_{\otimes}(X)$, which implies the statement.
\end{proof}

\bibliographystyle{plain}
\bibliography{HigherTwisted}

\end{document}